\documentclass[english, 10pt]{amsart}

\usepackage{euscript,graphicx,epstopdf,amscd,amsgen,amsfonts,amssymb,latexsym,amsmath,amsthm,graphicx,mathrsfs,times,color,overpic,mathtools}

\usepackage{fourier}
\usepackage{bbm}
\usepackage[all]{xy}
\usepackage{enumitem}
\usepackage{xcolor}

\makeatletter
\let\save@mathaccent\mathaccent
\newcommand*\if@single[3]{%
  \setbox0\hbox{${\mathaccent"0362{#1}}^H$}%
  \setbox2\hbox{${\mathaccent"0362{\kern0pt#1}}^H$}%
  \ifdim\ht0=\ht2 #3\else #2\fi
  }
\newcommand*\rel@kern[1]{\kern#1\dimexpr\macc@kerna}
\newcommand*\widebar[1]{\@ifnextchar^{{\wide@bar{#1}{0}}}{\wide@bar{#1}{1}}}
\newcommand*\wide@bar[2]{\if@single{#1}{\wide@bar@{#1}{#2}{1}}{\wide@bar@{#1}{#2}{2}}}
\newcommand*\wide@bar@[3]{%
  \begingroup
  \def\mathaccent##1##2{%
    \let\mathaccent\save@mathaccent
    \if#32 \let\macc@nucleus\first@char \fi
    \setbox\z@\hbox{$\macc@style{\macc@nucleus}_{}$}%
    \setbox\tw@\hbox{$\macc@style{\macc@nucleus}{}_{}$}%
    \dimen@\wd\tw@
    \advance\dimen@-\wd\z@
    \divide\dimen@ 3
    \@tempdima\wd\tw@
    \advance\@tempdima-\scriptspace
    \divide\@tempdima 10
    \advance\dimen@-\@tempdima
    \ifdim\dimen@>\z@ \dimen@0pt\fi
    \rel@kern{0.6}\kern-\dimen@
    \if#31
      \overline{\rel@kern{-0.6}\kern\dimen@\macc@nucleus\rel@kern{0.4}\kern\dimen@}%
      \advance\dimen@0.4\dimexpr\macc@kerna
      \let\final@kern#2%
      \ifdim\dimen@<\z@ \let\final@kern1\fi
      \if\final@kern1 \kern-\dimen@\fi
    \else
      \overline{\rel@kern{-0.6}\kern\dimen@#1}%
    \fi
  }%
  \macc@depth\@ne
  \let\math@bgroup\@empty \let\math@egroup\macc@set@skewchar
  \mathsurround\z@ \frozen@everymath{\mathgroup\macc@group\relax}%
  \macc@set@skewchar\relax
  \let\mathaccentV\macc@nested@a
  \if#31
    \macc@nested@a\relax111{#1}%
  \else
    \def\gobble@till@marker##1\endmarker{}%
    \futurelet\first@char\gobble@till@marker#1\endmarker
    \ifcat\noexpand\first@char A\else
      \def\first@char{}%
    \fi
    \macc@nested@a\relax111{\first@char}%
  \fi
  \endgroup
}
\makeatother

\definecolor{forestgreen}{rgb}{0.0, 0.27, 0.13}
\definecolor{darkslateblue}{rgb}{0.28, 0.24, 0.55}
\definecolor{blue1}{HTML}{0071A4}
\definecolor{green1}{HTML}{007D1B}

\usepackage[colorlinks=true,
backref=page,
citecolor=green1,
linkcolor=blue1,
]{hyperref}

\numberwithin{equation}{section}

\theoremstyle{definition}
\newtheorem{definition}{Definition}[section]

\theoremstyle{plain}
\newtheorem{assumption}{Assumption}

\newtheorem{lemma}[definition]{Lemma}

\newtheorem{theorem}[definition]{Theorem}
\newtheorem*{theorem*}{Theorem}

\newtheorem{manualtheoreminner}{Theorem}
\newenvironment{manualtheorem}[1]{%
  \renewcommand\themanualtheoreminner{#1}%
  \manualtheoreminner
}{\endmanualtheoreminner}

\newtheorem{proposition}[definition]{Proposition}

\theoremstyle{plain}
\newtheorem{corollary}[definition]{Corollary}

\theoremstyle{definition}
\newtheorem{remark}[definition]{Remark}

\theoremstyle{definition}
\newtheorem{example}[definition]{Example}

\theoremstyle{plain}

\theoremstyle{definition}

\makeatother

\newcommand{\diam}{\operatorname{diam}}
\newcommand{\vdim}{\operatorname{vdim}}
\newcommand{\Leb}{\operatorname{Leb}}

\newcommand{\Card}{\operatorname{Card}}
\newcommand{\eps}{\varepsilon}

\newcommand{\quand}{\quad\text{ and } \quad}
\newcommand{\cY}{{\mathcal Y}}
\newcommand{\cV}{{\mathcal V}}
\newcommand{\cA}{{\mathcal A}}

\newcommand{\cB}{{\mathcal B}}
\newcommand{\cC}{{\mathcal C}}

\newcommand{\cM}{{\mathcal M}}
\newcommand{\cL}{{\mathcal L}}
\newcommand{\cP}{{\mathcal P}}
\newcommand{\cQ}{{\mathcal Q}}
\newcommand{\cS}{{\mathcal S}}
\newcommand{\cT}{{\mathcal T}}

\newcommand{\sG}{\mathscr{G}}

\newcommand{\R}{{\mathbb R}}

\newcommand{\N}{{\mathbb N}}

\newcommand{\PP}{{\mathbb P}}

\newcommand{\eqdef}{\coloneqq}

\newcommand{\bN}{\N}
\newcommand{\bR}{\R}
\newcommand{\vrt}{{\widebar{\tau}}}
\newcommand{\p}{{\widebar{p}}}
\newcommand{\intr}{\operatorname{int}}

\numberwithin{equation}{section}

\begin{document}

\title[Generic points for non-statistical dynamical systems]{Hausdorff dimension of sets of generic points for non-statistical dynamical systems}

\author[D. Coates]{Douglas Coates}
\address{Institute of Mathematics, Federal University of Rio de Janeiro, Av. Athos da Silveira Ramos 149, Cidade Universit\'aria - Ilha do Fund\~ao, Rio de Janeiro 21945-909, RJ, Brazil}
\email{\href{mailto:coates@im.ufrj.br}{coates@im.ufrj.br}}

\author[K. Gelfert]{Katrin Gelfert}
\address{Institute of Mathematics, Federal University of Rio de Janeiro, Av. Athos da Silveira Ramos 149, Cidade Universit\'aria - Ilha do Fund\~ao, Rio de Janeiro 21945-909, RJ, Brazil}
\email{\href{mailto:gelfert@im.ufrj.br}{gelfert@im.ufrj.br}}
\thanks{K.~Gelfert have been supported [in part] by 
CAPES -- Finance Code 001 and  CNPq grant 305327/2022-4  
This study was funded by FAPERJ – Carlos Chagas Filho Foundation for Research Support of the State of Rio de Janeiro, Processes SEI E-26/200.371/2023,  
    E-16/2014 INCT/FAPERJ, and 
    E-26/200.027/2025 and 200.028/2025, 
and by the FAPESP Carlos Chagas Filho Foundation for Research Support of the State of São Paulo Research Foundation grant 2022/16259-2.
The authors thank their home institutions for their hospitality while preparing this paper.
The authors thank Gabriela Estevez for pointing out the elegant inequality in~\eqref{eq:abcd-clever} which made many of our arguments less tedious. The authors would also like to also thank Daniel Smania for stimulating conversations at the onset of this project and for showing us~\cite{MorSma2014} which inspired some of the ideas in Section~\ref{sec:approx}. A big thank you also to Fabio Tal for giving the key idea which lead to the removal of the assumption that $\cC$ is convex in Theorem~\ref{thm:dim-H-C}.
}

\subjclass[2020]{%
}

\begin{abstract}
  We consider one dimensional maps with several neutral fixed points that do not admit any physical measures. We show that there is simplex of measures so that every measure in this simplex has a basin which has full Hausdorff dimension.
\end{abstract}

\maketitle
\section{Introduction}

We study generic points for non-ergodic measures for certain non-statistical maps of the interval. Our results apply to large class of interval maps admitting $d \ge 2$ equally sticky neutral fixed points. We first present our results in the context of a simple example. The full description will be given in Section \ref{sec:setup}.

Following \cite{SerYan:2019}, consider the transformation $f \colon [0,1]\to [0,1]$ given by
\begin{equation}\label{eq:def-f}
  f(x) \eqdef
  \begin{cases}
  x + c_1 x^3,& x \in [0,1/3], \\
  x + c_2 ( x - \frac{1}{2})^3,& x \in (1/3, 2/3), \\
  x + c_3 ( x -  1 )^3,& x \in [2/3, 1],
  \end{cases}
\end{equation}
where $c_1 = c_3 = 18$ and $c_2 = 72$, so that $f$ is full-branched (see Figure~\ref{fig:map_example}). The map \eqref{eq:def-f} has derivative $f'(\cdot) > 1$, except at the neutral fixed points $\xi_1 =0$, $\xi_2 = 1/2$, and $\xi_3 =1$. This map admits a unique\footnote{Unique up-to scaling by a positive constant.} $\sigma$-finite absolutely continuous ergodic invariant measure $\mu$ which gives infinite mass to the unit interval.

\begin{figure}[h] 
 \begin{overpic}[scale=.27]{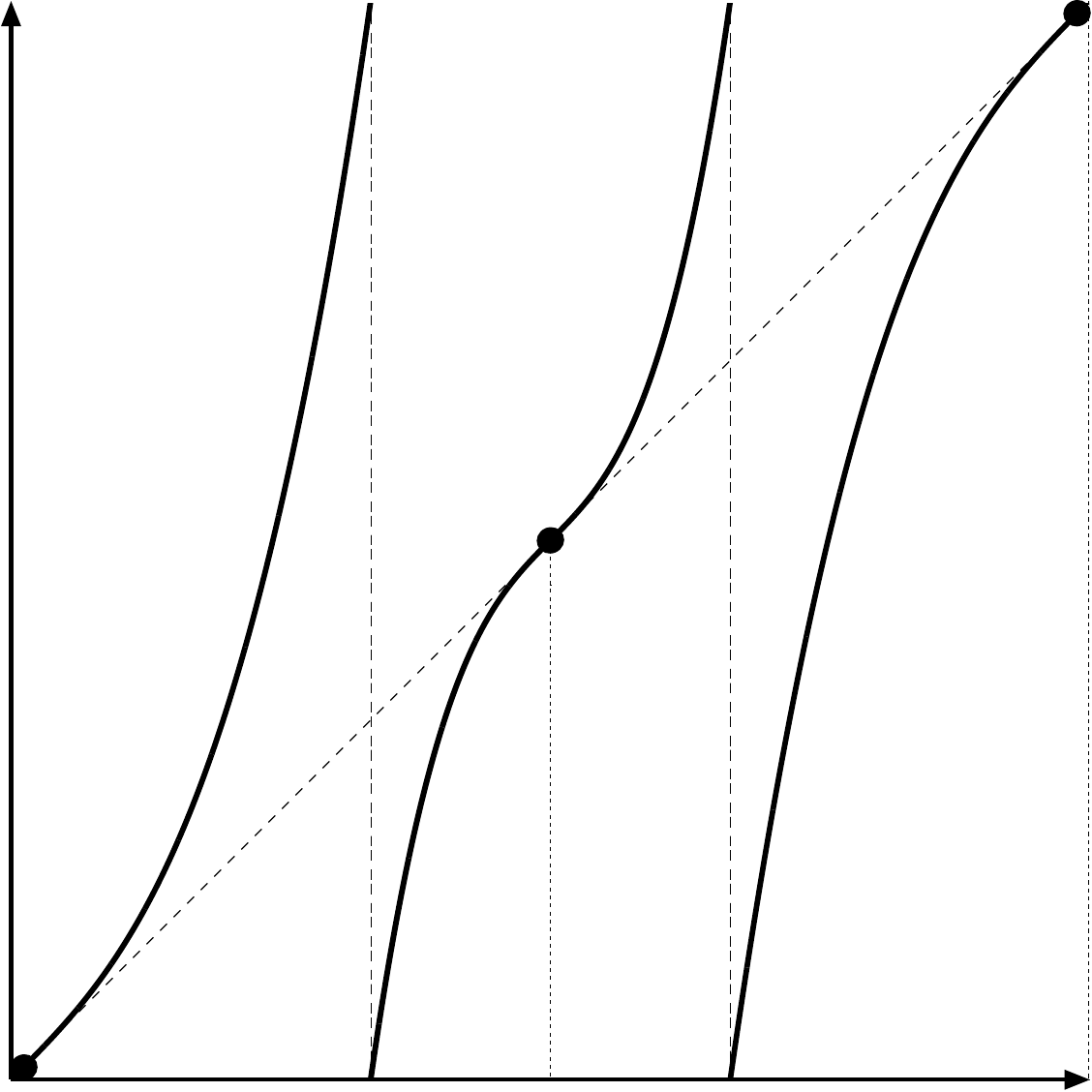}
 	\put(-2,-5){\small{$\xi_1$}}
 	\put(32,-6){\small{$\frac13$}}
 	\put(65,-6){\small{$\frac23$}}
 	\put(48,-5){\small{$\xi_2$}}
 	\put(97,-5){\small{$\xi_3$}}
 	\put(-5,0){\small{$0$}}
 	\put(-5,97){\small{$1$}}
\end{overpic}
\caption{The map $f$ given in~\eqref{eq:def-f}.} 
 \label{fig:map_example}
\end{figure}
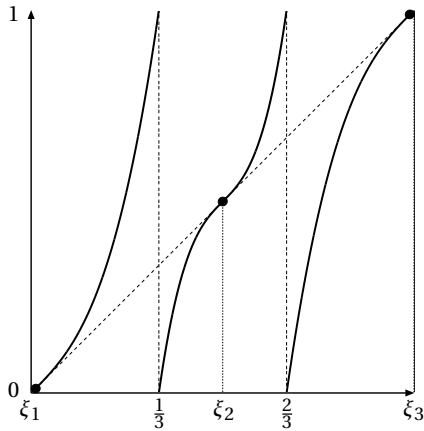

For $x \in X$, let $\delta_x$ denote the Dirac delta measure at $x$. By~\cite{CoaMelTal:2024}, the map $f$ given by \eqref{eq:def-f} is \emph{non-statistical} in the sense that for Lebesgue almost every $x$, the \emph{empirical measures}
\[
  e_n (x) \coloneqq \frac{1}{n} \sum_{ k = 0 }^{ n - 1 } \delta_{ f^k x }
\]
do not converge  in the space of Borel probability measures in the weak$\ast$ topology, as $n\to\infty$. Even more, for Lebesgue-almost every point $x$, these measures distribute along the simplex $\pi \cS$ of invariant measures supported at the fixed points, where
\[
  \cS \eqdef \{ \p \in [0,1]^3 \colon p_1 + p_2 + p_2 = 1\},
  \quand
  \pi \colon \cS \to \cM (f), \quad
  \pi ( \p ) = \nu_{\p} \coloneqq \sum_{ j = 1 }^{3} p_i \delta_{\xi_{i}}
  ,
\]
and where $\cM(f)$ denotes  the space of $f$-invariant Borel probability measures.
Indeed, letting
\begin{equation}\label{eqVx}
	\cV (x) \coloneqq \big\{ \text{limit points of } (e_n (x))_{n\in\bN} \big\},
\end{equation}
by \cite[Theorem A]{CoaMelTal:2024}, one has
\begin{equation}\label{shown}
  \cV(x)
  = \pi\cS\quad\text{ for Lebesgue almost every }x.
\end{equation}

Given an $f$-invariant probability measure $\nu$, consider its set of \emph{generic points} or \emph{basin of attraction},
\begin{equation}\label{eqGnu}
  \sG ( \nu ) \eqdef \{ x \colon e_{ n } (x) \to \nu \text{ in the weak$\ast$ topology as }n\to\infty\}. 
\end{equation}
By~\cite{CoaMelTal:2024}, the basin of attraction of any measure in $\pi\cS$  is negligible in terms of the Lebesgue measure
\begin{equation}\label{eq:leb-0}
  \Leb   ( \sG ( \nu ) ) = 0, \quad \text{ for all } \nu \in \pi\cS
  .
\end{equation}
Moreover, by~\cite[Theorem 4.1]{PfiSul:07} (see also \eqref{gen-K} below), it is also negligible in terms of the topological entropy%
\footnote{Here we use the concept of topological entropy $h_{ \rm top }   (f,\cdot)$ on general (not necessarily compact) subsets following \cite{Bow:73}.} of $f$
\begin{equation}\label{eq:top-0}
  h_{ \rm top }   (f, \sG ( \nu ) ) = 0, \quad \text{ for all } \nu \in \pi\cS.
\end{equation}
Our main result, applied to the map \eqref{eq:def-f}, tells us that despite~\eqref{eq:leb-0} and~\eqref{eq:top-0} these basins of attraction are in fact maximal in terms of Hausdorff dimension.
 
\begin{theorem}\label{thm:theorem1}
For every $\p \in \cS$,
$
  \dim_{\rm H } \sG ( \nu_\p )   = 1.
$
\end{theorem}

Note that the set $\cV(x)$ defined in \eqref{eqVx} is a nonempty
closed connected subset of $\cM(f)$. Our second result tells us that for \emph{any} closed connected 
subset of $\pi\cS$ the conclusion of the above theorem remains true.

\begin{theorem}\label{thm:theorem2}
For every closed connected subset $\cC \subset \cS$,
$
  \dim_{\rm H } \{ x \colon \cV(x) = \pi\cC \} = 1.
$
\end{theorem}

Theorems~\ref{thm:theorem1} and~\ref{thm:theorem2} fall within the scope of multifractal analysis, which investigates the ``size'' of certain level sets of asymptotic quantifiers such as Birkhoff averages and local dimensions. In this type of analysis, two sorts of behaviour are frequently studied: when a limit, say the average of a function $\phi\colon X\to\bR$, exists,
\begin{equation}\label{eq:levelset}
	\cL_\phi(\alpha)
	\eqdef \Big\{x\colon \lim_{n\to\infty}\frac1n\sum_{k=0}^{n-1}\phi(f^k(x))=\alpha\Big\},
\end{equation}
or when a limit does not exist and the full range of irregular behaviour with lower and upper limits $\alpha<\beta$ is investigated:
\begin{equation}\label{eq:levelsetdiv}
	\cL_{\phi,\rm irr}(\alpha,\beta)
	\eqdef\Big\{x\colon \liminf_{n\to\infty}\frac1n\sum_{k=0}^{n-1}\phi(f^k(x))=\alpha,
				\limsup_{n\to\infty}\frac1n\sum_{k=0}^{n-1}\phi(f^k(x))=\beta\Big\}.
\end{equation}
For a map $f$ being a full shift and a function $\phi$ being H\"older continuous, the sets \eqref{eq:levelset} and \eqref{eq:levelsetdiv} can be characterized in terms of their Hausdorff dimension and topological entropy, and studied in terms of certain Gibbs measures which indeed govern the decay of deviations from the expected asymptotic average. See, for example, \cite{BarPesSch:97} for references to some classical results.

The investigation of entropy of \eqref{eq:levelset} and \eqref{eq:levelsetdiv}  is  done in terms of certain rate functions that govern the decay of the distribution of Birkhoff averages over space. This type of analysis is also referred to as level-1 large deviation results (see \cite{Ell:85,You:90} for details on rate functions and \cite{TakVer:03} for applications to multifractal analysis). 

Here, we study level-2 deviations, focusing on empirical \emph{measures} and their distributions.  The study of divergence points such as in \eqref{eq:levelsetdiv}, from the point of view of distributional measures \eqref{eqGnu} and quantification in terms of entropy, is given, for example, in \cite{ErcKueLin:05}. In a similar spirit, for dynamical systems satisfying the specification property, \cite{TakVer:03} studies the entropy of Birkhoff averages-level sets, with a particular focus also on the Manneville-Pomeau map.  As a precursor to \cite{TakVer:03,ErcKueLin:05},  \cite{Bow:73} showed that for a continuous map $f$ on a compact metric space $X$ the entropy of $f$ on the set of generic points \eqref{eqGnu} of an \emph{ergodic} probability measure $\nu$ is 
\[
	h_{\rm top}(f,\sG(\nu))
	= h_\nu(f).
\]
Non-generic points can also be described. 
By \cite[Theorem 4.1]{PfiSul:07}, for any nonempty connected compact set $\mathcal{C}\subset\cM(f)$ in the space of invariant Borel probability measures,
\begin{equation}\label{gen-K}
	h_{\rm top}(f,\{x\colon \cV(x)=\mathcal{C}\})\le \inf\{h_\mu(f)\colon\mu\in \mathcal{C}\}.
\end{equation}
The equality in \eqref{gen-K} also holds, but under stronger hypotheses, see \cite{PfiSul:07}. 
In the setting of Theorem~\ref{thm:theorem1}, one can show that~\eqref{gen-K} implies \eqref{eq:top-0}.

The study of Hausdorff dimension of level sets \eqref{eq:levelset} and \eqref{eq:levelsetdiv} usually requires some type of hyperbolicity as the results collected in \cite{BarPesSch:97} illustrate. In  \cite{GelRam:09,JohJorObePol:10}, the authors study  nonuniformly hyperbolic interval maps with indifferent fixed points. These maps can be considered as on the boundary to hyperbolicity as they have the specification property and have a finite number of ergodic non-hyperbolic measure supported on parabolic fixed points.  The studies in \cite{GelRam:09,JohJorObePol:10} follow an approach of \emph{bridging measures} based on ``exhausting'' the nonuniformly hyperbolic part of dynamics by ``uniformly hyperbolic sub-dynamics'', see also \cite{BarSch:00}. Theorems~\ref{thm:theorem1} and~\ref{thm:theorem2} focus on empirical measures and the \emph{non-hyperbolic} part. The approach in \cite{GelRam:09} will also be a key step to prove the assertion about the dimension of \eqref{eqGnu} in Theorem \ref{thm:theorem1}.

The phenomenon \eqref{shown} for a map with just two parabolic fixed points was shown in \cite{AarThaZwe:05,CoaLuz:24,BarCha:}. More generally, non-statistical behavior was observed in other contexts. For example, in the quadratic family \cite{HofKel:95} prove that there are uncountably many parameters which correspond to a map $f$ which has ``maximal oscillation'' in the sense that the empirical measures accumulate on the full simplex of $f$-invariant Borel probability measures. In \cite{Tal:22}, the maximal oscillation-property was shown to hold generically for a certain class of rational maps. Finally, \cite{AndGui:22} establish an analogous version of \eqref{shown} for certain irrational linear flows of the two-torus (see also the references therein for further results in this direction).

The paper is organised as follows. In Section~\ref{sec:setup} we introduce the more general setting which we study (for which \eqref{eq:def-f} is a particular example)  and state our main result, Theorem~\ref{thm:dim-H-C}. Theorems~\ref{thm:theorem1} and~\ref{thm:theorem2} are special cases of Theorem~\ref{thm:dim-H-C}. Section \ref{sec4} provides details on repellers and return times. In Sections~\ref{sec:coding},~\ref{sec:approx}, and~\ref{sec:bridging} we prove Theorem~\ref{thm:dim-H-C}. In Section~\ref{sec:examples} we provide further examples to which Theorem~\ref{thm:dim-H-C} applies.

\section{Assumptions and statement of results}\label{sec:setup}

\subsection{Assumptions}

We recall some classical definitions and make precise the assumptions we use throughout the paper.
Throughout, we let $X$ denote either the circle $\mathbb{S}^1$ or the interval $[0,1]$.
We say that a non-singular transformation $f\colon X\to X$ is \emph{Markov} if there exists a countable partition (modulo $\Leb$) $\cP$  of $X$ into closed sub-intervals such that the restriction of $f$ to any partition elements is a bijection onto a union of partition elements. A Markov map $f$ with partition $\cP$ is \emph{topologically mixing} if for every $a,b \in \cP$ there exists an $N\in\bN$ such that $f^n (a) \cap b \neq \emptyset$ for every $n \geq N$. 

\begin{assumption}\label{asm:m}
$f\colon X \to X$ is a Markov map with respect to a partition $\cP$.
\end{assumption}

Throughout let $\tau\colon X \to \N$ denote the \emph{first hit-time}
\[
	\tau(x) \colon X \to \N,\quad
	\tau (x) \eqdef \min \{ n \in\bN \colon f^n (x) \in Y \}.
\]	
For a subset $Y$ which is a union of elements in $\cP$ we define the \emph{first return map} 
\[
	F\colon Y \to Y,\quad 
	F(x) \eqdef f^{\tau(x)}(x).
\]	 
We say that $F$ is a $C^2$ \emph{uniformly expanding Gibbs-Markov map} if there exists a partition $\cQ$ of $Y$ which is a refinement of  $\cP \cap Y$  with respect to which $F$ is Markov, and if the following additional properties hold:
  \begin{itemize}[leftmargin=0.5cm ]
  \item \textbf{(finite images)} $\Card \{ F (a) \colon a \in \cQ \} < \infty$, that is, there exist $Y_1,\ldots, Y_L$ so that each $Y_i$ is a union of elements of $\cQ$ and $\{F(a)\colon a\in \cQ\} = \{ Y_1, \ldots, Y_L\}$.
  
  \item \textbf{(bounded distortion)} There exist $\theta \in (0,1)$ and $C > 0$ so that the function $\log \frac{ d m } { dm \circ F }$ is $d_\theta$-Lipschitz on elements of $\cQ$,
  where $d_{\theta}$ is the metric $d_{\theta} (x,y) \coloneqq \theta^{s(x,y)}$ and 
  \[
    s(x,y) \eqdef \inf \big\{ n \in\bN_0 \colon F^n(x), F^n(y) \text{ lie in different elements of }\cQ \big\}
    .
  \]
  \item \textbf{($C^{ 2 }$ branches)} For every $a\in\cQ$ the map $F|_{\intr(a)}\colon \intr(a) \to F(\intr(a))$ is a $C^{ 2 }$ diffeomorphism\footnote{Notice that we do not necessarily assume that $F$ has $C^{ 2 }$ extension to the closure of $a$.}.
  
  \item \textbf{(uniform expansion)} There exists $\lambda > 1$ such that $\inf_{ x\in  Y} |F'(x)|> \lambda$.
\end{itemize}

\begin{assumption}\label{asm:gm}
  There exists a union of partition elements $Y$ so that the first return map $F\colon Y \to Y$ is a topologically mixing $C^2$ uniformly expanding Gibbs-Markov map with respect to a partition $\cQ$.
\end{assumption}

\begin{remark}\label{rem:tau-constant-on-P}
  Up to taking a refinement of $\cQ$, we can assume that $\tau$ takes a finite and constant value on the interior of every $a\in \cQ$.
\end{remark}

\begin{remark}
  As the first return map $F$ is Gibbs-Markov, it preserves an absolutely continuous invariant probability measure $\mu_Y$  (see, for example,~\cite[Chapter 4]{Aar:1997}). This implies that $f$ preserves the  $\sigma$-finite $f$-invariant absolutely continuous ergodic measure $\mu$ given by
  \begin{equation}\label{eq:def-mu}
    \mu ( A ) \eqdef \sum_{ n = 0 }^{ \infty }  \mu_Y ( \{  \tau > n \} \cap F^{-n} A )
    .
  \end{equation}
  From \eqref{eq:def-mu} it follows that $\mu_Y = \mu|_{Y}$ and that
  \[
  \mu ( X ) < \infty\quad \text{ if and only if }\quad \int \tau \, d\mu_Y < \infty.
  \]
\end{remark}

\begin{assumption}\label{asm:dyn-sep}
  There exist constants $d \in\bN$, $\alpha \in (0,1)$, and $\gamma_1, \ldots, \gamma_d>0$, and a partition (mod $\Leb$) of $X \setminus Y$ into open sets $X_1,\ldots,X_d$ such that for every $i,j\in\{1,\ldots,d\}$
\begin{enumerate}[leftmargin=0.6cm ]
\item there is a single fixed point $\xi_j \in X_j$,
\item\label{itm:dyn-sep} for $i \neq j$, orbits cannot pass from $X_i$ to $X_j$ without first passing through $Y$, that is, if $f^{ n } (x) \in X_{ i }$ and $f^{ n + m }(x) \in X_{ j }$ for some $i \neq j$ and $m\in\bN$, then there is $k \in\{1,\ldots,m-1\}$ such that $f^{ n+k }( x) \in Y$,
\item\label{itm:reg-var} $\mu_Y ( \tau^{(j)} > n ) \sim \gamma_j n^{\alpha}$ where 
\[
    	\tau^{(j)}(y) \eqdef \Card \big\{ n \le \tau (y) \colon f^n (y) \in X_j\big\}.
\]	
\end{enumerate}
\end{assumption}

\begin{remark}\label{rem:accum}
  Notice that each connected component of $J_n \eqdef \{ x \in X \setminus Y \colon \tau(x) = n \}$   is the preimage of an interval (a connected component of $Y$) by one of the branches of $f$ (a $C^2$ diffeomorphism). Thus, the sets $J_n$ necessarily accumulate at fixed points of $f$, in the sense that 
  \[
    \bigcap_{ n\in\bN } \overline{ \bigcup_{ k > n }  J_k }
    = \{ \xi_1, \ldots, \xi_d \}.
  \]
\end{remark}

\begin{figure}[ht] 
 \begin{overpic}[scale=.6]{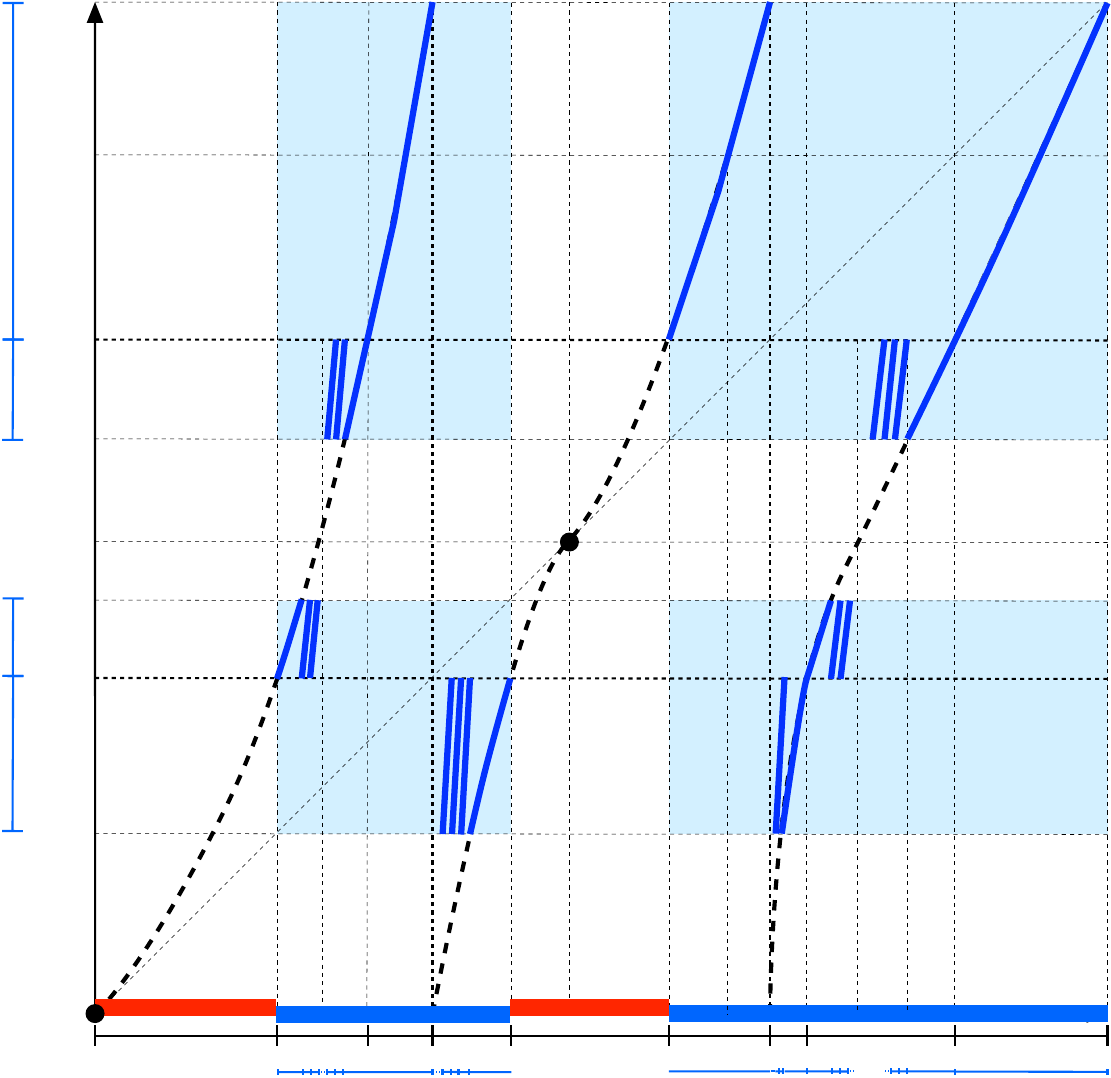}
 	\put(20,8){\small{$X_1$}}
 	\put(55,8){\small{$X_2$}}
 	\put(96,8){\small{$Y$}}
 	\put(102,-1){\small{$\cQ$}}
 	\put(-1,2){\rotatebox{90}{\small{$\{Y_1,\ldots,Y_4\}$}}}
 	\put(8,0){$\xi_1$}
 	\put(50,0){$\xi_2$}
\end{overpic}
  \caption{An example of a Markov map $f\colon[0,1]\to[0,1]$ (black dashed curve) with two parabolic fixed points $\xi_1,\xi_2$ and its Gibbs-Markov first return map $F\colon Y\to Y$ (blue curve) on the set $Y$ (union of blue intervals). See Example~\ref{ex:d+D}. The collection $\{Y_1,\ldots,Y_4\}$ of sets $\{F(a)\colon a\in\cQ\}$ cover $Y$, see Assumption~\ref{asm:gm}.}
 \label{fig.F}
\end{figure}
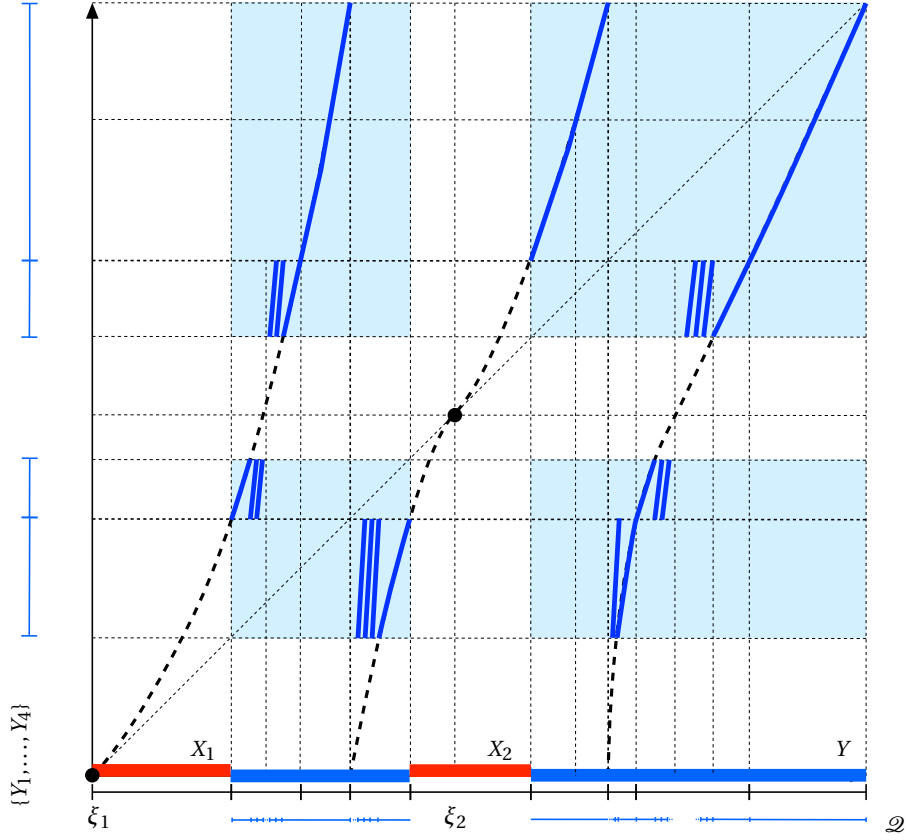

\begin{remark}\label{reg-var}
  Item~\eqref{itm:dyn-sep} in Assumptions~\ref{asm:dyn-sep} tells us that for any $x \in Y$, only one of the values $\tau^{(1)} (x), \ldots, \tau^{(d)}(x)$ can be non-zero. In particular,
  \begin{equation}
    \label{eq:dyn-sep}
    \tau = \tau^{(1)} + \cdots + \tau^{(d)} - 1
    ,
  \end{equation}
  and so
  \[
    \mu ( \tau > n) \sim \gamma n^{ \alpha }, \quad \text{where} \quad
    \gamma = \gamma_1 + \cdots + \gamma_d.
  \]
  In particular $\int \tau \, d\mu_Y = \infty$ and so $\mu( X ) = \infty$.
\end{remark}

\subsection{Statement of results}
\label{sec:statements}

Throughout the remainder of the article we assume that $f$ is a map satisfying Assumptions~\ref{asm:m}--\ref{asm:dyn-sep}. Compare also Figure \ref{fig.F}. We consider
\[
  \cS \eqdef \left\{ \p \in [0,1]^d \colon \sum_{ j = 1 }^{d} p_j = 1 \right\},
  \quand
  \pi \colon \cS \to \cM (f), \quad
  \pi ( \p ) = \nu_{\p} \coloneqq \sum_{ j = 1 }^{d} p_j\delta_{\xi_j}
  ,
\]
Denoting by $e_n(x) \coloneqq \sum_{ k = 0 }^{ n - 1 } \delta_{ f^k (x) }$ the $n$th empirical measure of $x \in X$, we consider  
\[
	\cV (x) \coloneqq \big\{ \text{limit points of } (e_n (x))_{n\in\bN} \big\}
\]
of weak$\ast$ limit points of the empirical measures.
The following is our main result.

\begin{manualtheorem}{A}\label{thm:dim-H-C}
 If Assumptions~\ref{asm:m}--\ref{asm:dyn-sep} hold, then for every closed connected set $\cC\subset \cS$,
  \[
    \dim_{\rm H} \{ x \colon \cV(x) =  \pi \cC \} = 1. 
  \]
\end{manualtheorem}

Considering the special case that $\cC = \{ \p \}$, one obtains:
\begin{corollary}
  \label{cor:dim-H-nup}
  If Assumptions~\ref{asm:m}--\ref{asm:dyn-sep} hold, then
  $\dim_{\rm H } \sG ( \nu_{\p} ) = 1$,
  for every $\p \in \cS$.
\end{corollary}

Sections~\ref{sec:coding}--\ref{sec:bridging-again} are devoted to proving Theorem~\ref{thm:dim-H-C}. It will be convenient to first prove Corollary~\ref{cor:dim-H-nup} which is a special case.

The first key step is to establish a coding of generic points in terms of the asymptotic behaviour of their return times. For this we will consider the joint behaviour of the return times $\tau^{(1)},\ldots,\tau^{(d)}$, and their Birkhoff sums under $F$. To this end we introduce
\[
	\vrt \eqdef ( \tau^{ (1) }, \ldots, \tau^{(d)} )\colon Y\to  \N^d,
\]
and let
\[
  \vrt_k \eqdef \sum_{j = 0}^{ k - 1 } \vrt \circ F^j,
  \quand
  \tau_k \eqdef \sum_{ j = 0 }^{ k - 1 } \tau \circ F^j,
\]
and define the $\tau^{(i)}_k$ analogously for $i=1,\ldots,d$.
In Section~\ref{sec:coding} we will show
\begin{equation}\label{eq:coding}
	e_n(x)\to\nu_\p\quad\Longleftrightarrow\quad
	\frac{\vrt_k}{\tau_k}(x)\to\p.
\end{equation}
Relation \eqref{eq:coding} reduces the problem of understanding the observables $e_n\colon X \to \cM(f)$ under the dynamics of $f$, to understanding the Birkhoff sums of the function $\vrt \colon Y \to \N^d$ under the dynamics of uniformly expanding Gibbs-Markov map $F$.

By Remark~\ref{rem:tau-constant-on-P}, $\vrt$ is piecewise constant. 
So, together with \eqref{eq:coding}, we obtain a symbolic description of $\mathscr{G} (\nu_{\p})$ (albeit one that requires an infinite alphabet) and can start building approximations of $\mathscr{G} ( \nu_{\p} )$.
Namely, we will construct collections $\cA$ of cylinder sets where the return times have a certain prescribed behaviour up to some finite time.
This will be done in such a way that: the limit points of $(\vrt_k/\tau_k)_k$ will all be close to $\p$
for points in the maximal invariant set of $\cA$; and, the dimension of this maximal invariant set is close to $1$. See Sections \ref{sec4} and~\ref{sec:approx}.

In Section~\ref{sec:bridging} we then consider a sequence $(\cA_i)_i$ of increasingly better approximations to $\mathscr{G}(\nu_\p)$ to build a set $\Gamma \subset \mathscr{G} ( \nu_\p )$. We conclude by arguing that $\dim_{\rm H} \Gamma = 1$.

Finally, in Section~\ref{sec:bridging-again} we explain how the proof of Corollary~\ref{cor:dim-H-nup} can be strengthened to prove Theorem~\ref{thm:dim-H-C}.

\section{Coding of  $\nu_{\bar p}$-generic points}\label{sec:coding}

In this section we will explore equivalent characterizations of the convergence of the empirical measures in terms of the return times $\vrt$ by relating $\cV(x)$ and the set
\[
  \cT(x) \eqdef \left\{ \text{limit points of the sequence } \left(\frac{\vrt_k }{\tau_k}(x)\right)_{k\in\bN} \right\} \subset [0,1]^d
  .
\]

\begin{theorem}\label{thm:coding}
 Let be $\cT(x)$ a closed connected subset of $\cS$ and suppose that
\begin{equation}\label{eqcondtaus}
	\lim_{k\to\infty}\left| \frac{\vrt_k}{\tau_k}(x) - \frac{\vrt_{k+1}}{\tau_{k+1}}(x) \right|=0,
\end{equation}
Then, $\cV(x) = \pi  (\cT(x) )$.
\end{theorem}

The above theorem will be proven at the end of this section. When $\cC$ contains a single point $\p$, Theorem \ref{thm:coding} reduces to the following result.

\begin{corollary}\label{cor:coding-single-point}
  For all $x\in X$ and $\p \in \cS$,
  \[
  	e_n(x)\to \nu_\p\quad\Longleftrightarrow\quad
  	\frac{\vrt_k}{\tau_k}(x)\to\p.
  \]
\end{corollary}

We begin with the following intermediate result.

\begin{proposition}\label{prop:subset}
  If $\mathcal{T}(x) \subset \cS$, then $\cV(x) \subset \pi\cS$.
\end{proposition}

To prove Proposition \ref{prop:subset} we first collect some preliminary results.
From the ergodic theorem for infinite measures  we know that
\begin{equation}
\begin{gathered}\label{eq:prop-K}
  \lim_{n\to\infty} e_n (x) (K) 
  = \lim_{n\to\infty} \frac{1}{n} \sum_{ i =  0 }^{ n - 1 } 1_{K} \circ f^i (x) = 0\\
  \text{ for every compact set } K \text{ with } K \cap \{ \xi_1, \ldots, \xi_d \} = \emptyset,
  \end{gathered}
\end{equation}
holds for Lebesgue-almost every%
\footnote{The ergodic theorem  (see \cite[Exercise 2.2.1]{Aar:1997}) here tells us that $e_n(x)(A) \to 0$  for $\mu$ almost every $x$ and for all $A$ with $\mu (A) < +\infty$. The fact that any compact set which does not contain any of the neutral fixed points has finite measure is an easy consequence of Remark~\ref{rem:accum} and~\eqref{eq:def-mu}.} $x$.
Recall that $e_n(x)(K)$ is just the proportion of time  up to time $n$ which the orbit of $x$ spends in $K$. So, if~\eqref{eq:prop-K} holds for some $x$, then its orbit must spend proportion $1$ of its time near the fixed points. This, in turn, suggests $\cV(x) \subset \pi\cS$. Indeed, we get the following result.

\begin{lemma}
  \label{lem:limit-points}
  For any $x\in X$, $\cV(x) \subset \pi\cS$ and \eqref{eq:prop-K} are equivalent.
\end{lemma}

\begin{proof}
  The Portmanteau theorem tells us that $\cV(x) \subset \pi\cS$ if and only if the only limit point of $(e_n(x)(A))_n$ is $0$ whenever $A\subset X$ is
  such that $ \partial A \cap \{ \xi_1 , \ldots, \xi_d\} = \emptyset$. The result quickly follows.
\end{proof}

From our discussion above, it is now clear that $\cV(x) \subset \pi \cS$ for Lebesgue-almost every $x$. (Indeed, note that by \cite{CoaMelTal:2024} the reverse inclusion $\pi\mathcal{S}\subset \mathcal{V}(x)$ also holds for Lebesgue-typical $x$).
We are interested though in the behaviour of non-typical points. Moreover, we want to encode condition~\eqref{eq:prop-K} using the symbolic dynamics of the induced map.  
For that, we introduce the condition
\begin{equation}\label{eq:birk-tau}
  \lim_{ k \to \infty } \frac{ \tau_k (x) }{ k } = +\infty.
\end{equation}

\begin{lemma}\label{lem:31-32}
  For any $x \in X$, conditions~\eqref{eq:prop-K} and~\eqref{eq:birk-tau} are equivalent.	
\end{lemma}

\begin{proof}
Assume \eqref{eq:prop-K}. Notice that
\[
  \frac{ k }{ \tau_k (x) } = e_{ \tau_k (x) } (x) (Y).
\]
Hence, \eqref{eq:birk-tau} follows immediately from \eqref{eq:prop-K} by taking $K = Y$. 

Now assume \eqref{eq:birk-tau}. As $K \subset X$ is a compact set which does not contain any of the neutral fixed points we know from Remark~\ref{rem:accum} that $\sup_{ K \setminus Y } \tau < \infty$.
  Setting 
  \[
  	N \eqdef \sup_{ x \in K \setminus Y } \tau (x),
\]	 
that number of entrances to $K$ between successive returns to $Y$ is at most $N$, that is,
  \[
    \sup_{ k \in\bN } \sum_{ i = \tau_{ k - 1 }(x) + 1 }^{ \tau_{ k } (x) - 1 } \mathbbm{1}_{ K } \circ f^{ i } (x) \le \sup_{ x \in K \setminus Y } \tau (x) = N < \infty.
  \]
  For each $n \in \N$ choose $k_{ n } \in \N$ such that
\[
    \tau_{ k_{ n } } (x) \le n < \tau_{ k_{ n }  + 1 } (x)
    .
\]
  By the definition of $k_{ n }$ and $N$ we have
  \[
    e_{ n }(x) (K) = \frac{1}{n} \sum_{ i = 0 }^{  n - 1 } \mathbbm{1}_{ K } \circ f
    ^{ i } (x) \le \frac{ k_{ n } N + N }{ \tau_{ k_{ n } } (x) } \to 0
  \]
  by our assumption on $x$. This proves \eqref{eq:prop-K}.
\end{proof}

\begin{proof}[Proof of Proposition~\ref{prop:subset}]
By Lemmas~\ref{lem:limit-points} and~\ref{lem:31-32}
it is enough to show $\cT(x) \subset \cS$ implies that~\eqref{eq:birk-tau}. We first claim that
  \begin{equation}\label{eq:lim-Sk/tauk}
    \lim_{k \to \infty} \frac{ \sum_{ j = 1 }^{d} \tau_k^{(j)}(x) }{ \tau_k (x)} = 1
    .
  \end{equation}
  From~\eqref{eq:dyn-sep} we know that $\sum_{ j = 1 }^{d} \tau_k^{(j)}(x) /  \tau_k (x) \le d$. So the sequence in the limit in \eqref{eq:lim-Sk/tauk} is bounded and hence we can choose a subsequence $(k_\ell)_\ell$ so that 
\[
  \lim_{\ell \to \infty }\frac{\sum_{ j = 1 }^{d} \tau_{ k_\ell }^{(j)}(x)}{  \tau_{ k_\ell } (x) }= c
 \]  
 for some $c \in [0,d]$. By the compactness of $[0,1]^d$, up to passing to a further subsequence, we can assume that $\lim_{\ell \to \infty } ( \vrt_{ k_\ell }/ \tau_{k_\ell } ) (x) = (p_1,\ldots,p_d)$. But as all the limit points $\cT (x)$ of $((\vrt_k/\tau_k)(x))_k$  are contained in $\cS$, we find that
  \[
    c = p_1 + \cdots + p_d = 1,
  \]
  which proves~\eqref{eq:lim-Sk/tauk}.

  Now, using~\eqref{eq:dyn-sep}, we see that
  \[
    \frac{ \sum_{ j = 1 }^{d} \tau_k^{(j)}(x) }{ \tau_k (x) } =
    \frac{ \tau_k (x) - k  }{ \tau_k(x) } = 1 - \frac{k}{ \tau_k (x)}.
  \]
  Thus,~\eqref{eq:lim-Sk/tauk} yields~\eqref{eq:birk-tau}. This proves the proposition.
\end{proof}

Define
\[
\widetilde \cT (x) \eqdef
\big\{ \text{limit points of } \big((e_n(x)(X_1),\ldots,e_n(x)(X_d))\big)_n \big\}
\subset [0,1]^d.
\]

\begin{lemma}\label{lem:tilde-T}
  If $\cV(x) \subset \pi\cS$, then $\cV(x) = \pi( \widetilde\cT(x) )$.
\end{lemma}

\begin{proof}
  Suppose that $\cV(x) \subset \pi\cS$.

As the observable $(\mathbbm{1}_{X_1}, \ldots, \mathbbm{1}_{X_d})$ is continuous at the continuity points of any measure in $\pi\cS$, the Portmanteau theorem tells us that if $e_{n_k}(x) \to \nu_\p$ for a subsequence $(n_k)_k$ and some $\p \in \cS$, then 
  \begin{equation}\label{eq:portman}
  \begin{aligned}
    \p 
    = \int (\mathbbm{1}_{X_1}, \ldots, \mathbbm{1}_{X_{d}} ) \, d\nu_{\p}
    &= \lim_{ k \to \infty }\int (\mathbbm{1}_{X_1}, \ldots, \mathbbm{1}_{X_{d}} ) \, d e_{n_k}(x)\\
    &= (e_{n_{k}}(x) (X_1), \ldots, e_{n_k}(x)(X_d)).
  \end{aligned}
  \end{equation}
  Thus,  $\cV(x) \subset \pi \widetilde \cT(x)$.
  
  Now suppose that $\lim_{k\to\infty}(e_{n_{k}}(x) (X_1), \ldots, e_{n_k}(x)(X_d)) = \p$ for a subsequence $(n_k)_k$. Reasoning as in~\eqref{eq:portman} we see that the only possible limit point of $(e_{n_k}(x))_k$ is $\nu_\p$. So, by compactness,
  $\lim_{k \to \infty }e_{n_k}(x) = \nu_\p$ which yields $\widetilde{\cT}(x) \subset \pi^{-1}\cV(x)$ and concludes the proof.
\end{proof}

\begin{proof}[Proof of Theorem~\ref{thm:coding}]
  Suppose that $\cT(x)$ is a closed connected subset of $\cS$.
  By Proposition~\ref{prop:subset}, $\cV(x) \subset \pi\cS$. Thus, from Lemma~\ref{lem:tilde-T}, it is enough to show that $\widetilde \cT(x) = \cT(x)$.  As
  \[
    \frac{ \vrt_k }{ \tau_k }  (x) = ( e_{\tau_k}(x) (X_1), \ldots, e_{\tau_k}(x) (X_d)),
  \] 
  it is clear that $\cT(x) \subset \widetilde{\cT}(x)$.
  
  By Assumption~\ref{asm:dyn-sep} item~\eqref{itm:dyn-sep} the orbit of $x$ spends all of its time in at most one of the sets $X_{ 1 }, \ldots, X_{ d }$ between any two successive entrances to $Y$.  Thus, for any given $j = 1,\ldots, d$, $k \in\bN$ and $\tau_k(x) \le \ell < \tau_{k + 1} (x)$, one has for $n = \tau_k(x) + \ell$ that either
  \[
    e_n (x)(X_j) =
    \frac{ \tau_{ k }^{ (j) } (x) + \ell}{ \tau_{ k } (x) + \ell },\quad \text{or} \quad
    e_n (x)(X_j) = \frac{ \tau_{ k }^{ (j) } (x) }{ \tau_{ k } (x) + \ell }
    ,
  \]
  and the former is increasing\footnote{Notice that the derivative of $t \mapsto \frac{ a + t }{ b + t }$ is $\frac{ b - a }{ ( b + t )^{ 2 }}$ which is non-negative whenever $b \ge a$.} in $\ell$ and the latter is decreasing in $\ell$. Thus, the sequence $(e_n(X_j))_n$ is strictly monotonic on each time interval $[\tau_{k}(x), \tau_{k+1}(x)] \cap \N$.
  This tells us that 
  \begin{equation}\label{eq:tau-ineq}
  	 \min\left\{\frac{\tau_k^{(j)}(x)}{\tau_k(x)},\frac{\tau_{k+1}^{(j)}(x)}{\tau_{k+1}(x)}\right\}
  	\le e_n(x)(X_j)
  	\le \max\left\{\frac{\tau_k^{(j)}(x)}{\tau_k(x)},\frac{\tau_{k+1}^{(j)}(x)}{\tau_{k+1}(x)}\right\}.
  \end{equation}

  Suppose that \eqref{eqcondtaus} is satisfied and let
  \[
  	\varepsilon_k
  	\eqdef \left|\frac{\tau_k^{(j)}(x)}{\tau_k(x)}-\frac{\tau_{k+1}^{(j)}(x)}{\tau_{k+1}(x)}\right|.
  \]
  Hence, as $\cT(x)$ are the limit points of $(\tau_k^{(j)}(x)/\tau_k(x))_k$, for all $k$ sufficiently large we get
  \[
  	\frac{\tau_k^{(j)}(x)}{\tau_k(x)}\in B_{2\varepsilon_k}(\cT(x)).
  \]
  Thus, we get that 
  \[
  	((e_n(x)(X_1),\ldots,e_n(x)(X_d))
  	\in B_{3\varepsilon_k}(\cT(x)),
  \]
  which implies that $\widetilde\cT(x)\subset\cT(x)$.
  This concludes the proof.
\end{proof}

\section{Auxiliary results on repellers and return times}\label{sec4}

In this section, we collect some auxiliary results on ``sub-repellers'' for some (high enough) iterate $F^n$ of our initial map $F$.

\subsection{Cylinders and associated maximal invariant sets}

Given $n\in\bN$, we let
\[
	\cQ^{ n } \eqdef \bigvee_{ i = 0 }^{  n - 1 } F^{ -i } \cQ
\]	 
denote the refinement of $\cQ$ under $F$ into $n$-cylinders. Write 
\[
	\cQ^{ * } \eqdef \{ a \in \cQ^{ n } \colon n \in \N_0 \}
\]	 
for the collection of all cylinders.
Given $a \in \cQ^{n}$ and $b\in \cQ^{*}$ we  write $ab \eqdef a \cap F^{-n} b$.

We denote by $|a|$ the length of the interval $a\in\cQ^\ast$.

Unless there is a risk of misunderstanding, given a collection of $n$-cylinders $\cA \subset \cQ^n$, we use the symbol $\cA$ also to denote the union of all partition elements which $\cA$ contains. For $\cA \subset \cQ^n$ consider the associated maximal invariant set 
\begin{equation}\label{eq:def:max-inv-set}
  \Lambda ( \cA ) \coloneqq \{ x \colon F^{kn} (x) \in \cA \text{ for all } k \in\bN_0  \}.
\end{equation}
Given any set $A \subset X$ we denote by $\cQ_{A}^\ell$ the set of $\ell$-cyclinders whose interiors intersect $A$,
\[
  \cQ^{\ell}_A \eqdef
  \big\{ a \in \cQ^\ell \colon \intr(a) \cap A \neq \emptyset \big\}
  .
\]
\begin{remark}\label{rem:ell-cover-lambda}
  Notice that if $\cA \subset \cQ^n$ then for every $k\in\bN_0$ and $s\in\{0,\ldots, n - 1\}$,
  \[
    \cQ_{\Lambda(\cA)}^{kn + s}
    = \big\{ a_1 \ldots a_k b \colon 
    a_1, \ldots, a_k \in \cA,
    \text{ and } b \in \cQ^s_{\cA}  \big\}
    .
  \]
\end{remark}

\subsection{Geometric properties of repellers formed from maximal invariant sets}\label{secprelimrepeller}

Given a finite collection $\cA \subset \cQ^n$, we define the \emph{virtual dimension} $\vdim \Lambda ( \cA ) $ of its maximal invariant set $\Lambda ( \cA )$ defined in~\eqref{eq:def:max-inv-set} by setting
\begin{equation}\label{defvDim}
	\text{ $\vdim\Lambda(\cA)$ is the unique $s > 0$ satisfying $\sum_{a\in\cA}|a|^s = 1$}
  .
\end{equation}
By Assumption~\ref{asm:gm}, the map $F$ is a $C^{ 2 }$ uniformly expanding Markov map with bounded distortion. Hence, there exist a $\lambda > 1$ and a constant $D > 0$ such that for every $n\in\bN$,
\begin{equation}\label{eq:distotion-expansion}
  	\inf_{x\in Y} \left|(F^n)'(x)\right| \ge \lambda^n
	\quad\text{ and }\quad
	\sup_{a\in\cQ^n}\sup_{ x,y\in a } \log \frac{| (F^n)' (x)|}{| (F^n)' (y) |} \le D.
\end{equation}
In particular, map $g = F^n|_{\Lambda (\cA )}$ is an \emph{expanding repeller} in the sense of~\cite{PrzUrb:10}. That is, equipping $Y$ with the topology induced by the separation time distance, $\Lambda (\cA)$ is a \emph{repeller} for $g$:
\begin{itemize}
  \item $g(\Lambda (\cA)) = \Lambda ( \cA )$;
  \item $g|_{\cA}$ is continuous;
  \item $\cA \subset Y$ is open and $\Lambda ( \cA ) = \bigcap_{ k = 0 }^{ \infty } g^{-k} \cA$;
  \item and, $g$ is \emph{expanding}: $|(g^k)'(x)| \ge \lambda^{nk}$ for all $x\in \Lambda ( \cA)$.
\end{itemize}

\begin{remark}\label{rem:dist-exp}
  For further reference, let us also remark that~\eqref{eq:distotion-expansion} implies that,
  \begin{equation}\label{smallLeb}
    D^{-1}\lambda^{-n}
    \le|a|\le \lambda^{-n}
    \quad\text{ for all }\quad n\in \N \text{ and }a\in\cQ^n,
  \end{equation}
  and
  \[
  	D^{-1} |a||b| < |ab| < D |a||b|
    \quad
    \text{whenever $a,b,ab \in \cQ^*$.}
  \]
\end{remark}

From \cite[proof of Theorem 3, Section 4.2]{PalTak:93} we then obtain.

\begin{lemma}\label{lem:approxvdim}
  Let $\cA \subset \cQ^n$ be a finite subcollection. If $F^n\colon \Lambda ( \cA ) \to \Lambda ( \cA )$ is topologically mixing then
  \[
    \big\lvert \dim_{\rm H}\Lambda(\cA)-\vdim_{\rm H}\Lambda(\cA)\big\rvert
    \le \frac{D}{n\log\lambda-D}
    ,
  \]
  where $D > 0$ and $\lambda > 1$ are as in~\eqref{eq:distotion-expansion}.
\end{lemma}

We say that a Borel probability measure $m$ is \emph{geometric} for $F^n|_{\Lambda ( \cA)}$
 if it is ergodic and invariant for $F^n$, $m ( \Lambda ( \cA ) )= 1$, and there exists a constant $E > 0$ so that
 \begin{equation}\label{eq:def-geom}
    \dim_{ \rm H } \Lambda ( \cA ) -  \frac{ E }{ \log 2r  }
    \le
    \frac{ \log m ( B_r (x) )  }{ \log 2r }
    \le
    \dim_{ \rm H } \Lambda ( \cA )  + \frac{ E  }{ \log 2r},
 \end{equation}
 for all $x \in \Lambda( \cA )$ and all $r > 0$.
The following is provided, for example, by \cite[Theorem 9.1.6]{PrzUrb:10} and \cite[Theorem 8.1.6]{PrzUrb:10}, by taking $m$ to be the Gibbs state corresponding to the function $-d\log|(F^n)'|$ (with respect to the mixing expanding repeller $F^n|_{\Lambda(\cA)}$, with $d= \dim_{\rm H}\Lambda(\cA)$.

\begin{lemma}\label{existence-m}
  For any $\cA\subset\cQ^n$ so that $F^n|_{ \Lambda (\cA)}$ is topologically mixing, 
   there exists a corresponding geometric measure $m$.
\end{lemma}

\subsection{Return times on cylinders}

 By Remark~\ref{rem:tau-constant-on-P}, the function $\tau$ (and hence the vector-valued function $\vrt$) is constant on the interior of each element of $\cQ$. Thus, for any $n\in\bN$ the value of $\vrt_{ n }(x)$ is completely determined by the $n$-cylinder which contains $x$. For $a\in \cQ^{ n }$, we will thus write $\vrt_{ n }(a)$ for the value of $\vrt_{ n }|_{ \intr(a) }$.

\begin{remark}\label{rem:simplerels}
  For $x,y,v,w >0$,
\begin{equation}\label{eq:abcd-clever}
 \text{ if }\quad \frac x y < \frac v w
    \quad\text{ then }\quad
    \frac{x}{y} \le \frac{ x + v }{ y + w } \le \frac{v}{w}.
\end{equation}
For $\widebar{x}, \widebar{y} \in \R^{ d }$ and $v,w \neq 0$ we have that 
  \begin{equation}\label{eq:abcd-trivial}
    \frac{ \widebar x + \widebar y }{ v + w } -
    \frac{\widebar x}{v}
    =
    \left(
        \frac{ \widebar y}{ w }
      -
        \frac{ \widebar x }{ v }
    \right)
    \frac{ w }{ v + w }.
  \end{equation}
\end{remark}

From Remark \ref{rem:simplerels} we obtain the following simple inequalities. Here, all equalities involving vectors in $\R^d$ are to be understood coordinate-wise. For $\widebar{x} \in \R^d$, let $|\widebar{x}| \coloneqq \max_{i = 1, \ldots, d } | x_i |$.

\begin{lemma}\label{lemlinalgebra}
  For every $a \in \cQ^n$ and $b \in \cQ^k$ such that $ab\in\cQ^{n+k}$,
\begin{enumerate}
\item     \label{eq:n+m-with-n}
$\displaystyle \left|
      \frac{ \vrt_{ n + k } }{ \tau_{ n + k } } (ab) 
      - \frac{ \vrt_n }{ \tau_n } (a)
    \right|
    \leq
    \frac{2 \tau_k (b) }{ \tau_{n} (a) + \tau_k (b)}$,
\item      \label{eq:n+k-with-k}
$\displaystyle\left|
      \frac{ \vrt_{ n + k } }{ \tau_{ n + k } } (ab) 
      - \frac{ \vrt_k }{ \tau_k } (b)
    \right|
    \leq
    \frac{2 \tau_n (a) }{ \tau_{n} (a) + \tau_k (b)}$,
\item        \label{eq:n+k-with-min-max}
$\displaystyle
    \min\left\{
      \frac{ \vrt_{ n } }{ \tau_{ n } } (a) ,
      \frac{ \vrt_{ k } }{ \tau_{ k } } (b) 
    \right\}
    \le
    \frac{ \vrt_{ n + k } }{ \tau_{ n + k } } (ab) 
    \leq
    \max\left\{
      \frac{ \vrt_{ n } }{ \tau_{ n } } (a) ,
      \frac{ \vrt_{ k } }{ \tau_{ k } } (b) 
    \right\}$.    
\end{enumerate}  
\end{lemma}

\begin{proof}
  First notice that
  \[
    \frac{ \vrt_{ n + k } }{ \tau_{ n + k  } } (ab)
    = \frac{ \vrt_{ n } (a) + \vrt_{ k } (b) }{ \tau_{ n  } (a)  + \tau_{ k }(b) }
    .
  \]
  The inequality in item~\eqref{eq:n+k-with-min-max} then follows immediately from~\eqref{eq:abcd-clever}.   
  Also note that $|\vrt_j/\tau_{ j }| \le 1$. Hence, \eqref{eq:abcd-trivial} implies
\[
  \left|
	\frac{ \vrt_{ n + k } }{ \tau_{ n + k  } } (ab) -  \frac{ \vrt_n }{ \tau_n } (a)
  \right|
	=
  \left|
  \left( \frac{\vrt_{ k } (b) }{ \tau_{ k }(b)} - \frac{\vrt_{ n } (a) }{\tau_{ n  } (a) } \right) 
		\frac{ \tau_{ k }(b)}{\tau_{ n  } (a)  +  \tau_{ k }(b)}
  \right|
	\le 2\frac{ \tau_{ k }(b)}{\tau_{ n  } (a)  +  \tau_{ k }(b)}	
\]  
together with the analogous lower bound. This proves item~\eqref{eq:n+m-with-n}. The proof of item~\eqref{eq:n+k-with-k}  is analogous.
\end{proof}

\section{Approximation by repellers}\label{sec:approx}
%

Recall the big image property for $F$  from Assumption~\ref{asm:gm} stating that $\{F(a)\colon a\in\cQ\}=\cY \eqdef \{ Y_1, \ldots, Y_L \}$ for some $L\in\bN$. It hence follows that 
\begin{equation}\label{eq:def-im-part}
    \{ F^n (a) \colon a \in \cQ^n \}
    = \cY,
    \quad
    \text{for every }n \ge 1
    .
\end{equation}
In this section we will prove the following proposition.

\begin{proposition}\label{prop:approx-new}
  For every $\p \in \cS$ and every $\eps > 0$ there exist $n,N \in \N$, $\cA \subset \cQ^n$, and a Borel probability measure $m$ on $Y$ such that: 
  the images of $\cA$ under $F^n$ cover $Y$ and, in particular, $F^n ( \cA ) = \cY$;
   $m(\Lambda(\cA)) = 1$;
  and $m ( \Lambda(\cA) \cap Y_i ) > 0$ for each $i = 1, \ldots, L$.
   Moreover,
  \begin{equation}\label{eq:approx}
    \frac{ \vrt_\ell}{ \tau_\ell}  (a) \in B_\eps ( \p )
    \quand
    \frac{\log m ( a ) }{ \log |a| } \in B_{\eps} ( 1 ),
    \quad\text{for all } \ell > N \text{ and all } a \in \cQ_{\Lambda (\cA )}^\ell.
  \end{equation}
\end{proposition}

Throughout this section, we fix $\p\in\cS$ and $\varepsilon\in(0,1)$.
In Section~\ref{sec:const-An} we construct, for each $n \in \N$ sufficiently large, a collection of $n$-cylinders $\cA(n)$.
In Sections~\ref{sec:rettime-An} and~\ref{sec:geom-An} we show, for all $\ell \in  \N$ sufficiently large, that cylinders $a \in \cQ_{\Lambda(\cA(n))}^{\ell}$ satisfy the assertion~\eqref{eq:approx} of Proposition~\ref{prop:approx-new} with respect to a certain measure $m(n)$ on $\Lambda(\cA(n))$.

\subsection{Construction of $\cA(n)$}
\label{sec:const-An}

\begin{proposition}\label{prop:const-An}
  There exists $C\in(0,1)$ so that for every $n$ sufficiently large there is a finite collection $\cA(n) \subset \cQ^{ n }$ such that
  \begin{enumerate}
    \item\label{itm:1-def-An} $\cA(n) \subset \{ a \in \cQ^n \colon \vrt_{ n } (a) / \tau_{ n } (a) \in B_{ 3\eps/4 } ( \p ) \}$;

    \item\label{item2lem:def-An-new} $\Leb  (\cA (n) \cap Y_i)  > C$, for all $i = 1, \ldots, L$;

    \item\label{item3lem:def-An} $F^{ n } \cA(n) = \cY$.

  \end{enumerate}
\end{proposition}

We divide the proof of Proposition~\ref{prop:const-An} into several intermediate lemmas. 
First, we give the following uniform lower bound for the Lebesgue measure of the set of points $x$ where the ratio $(\vrt_{ n } / \tau_{ n }) (x)$ is close to $\p$.

\begin{lemma}\label{lem:uniform-lower-bound}
  There exists $C^{\prime} > 0$ such that
  \[
    \Leb \big\{x  \in Y \colon \vrt_{ n } (x) / \tau_{ n } (x) \in B_{ \eps/2 } (\p) \big\} > C^{\prime},
  \]
  for all $n$ sufficiently large.
\end{lemma}

\begin{proof}
  It is enough to prove that  there exists a random variable $\boldsymbol{Z}$ whose distribution $\boldsymbol Z_{ * } \PP$ on $[0,1]$ is equivalent to $\Leb$ and such that the random variables 
\[
  x \mapsto ( e_{ n } (x) ( X_{ 1 } ) , \ldots, e_{ n } (x) (X_{ d }))
\]   
converge in distribution to $\boldsymbol{Z}$ as $n\to\infty$. Indeed, if this is the case then
  \[
    \lim_{n \to \infty}\Leb \big\{x \colon \vrt_{ n } (x) / \tau_{ n } (x) \in B_{ \eps/2 } (\p) \big\}
    = \PP ( Z \in B_{\eps/2} (\p) ) > 0.
  \]
  The fact that there exists such a $\boldsymbol{Z}$ follows from~\cite[Corollary 4.2 and Theorem 3.3]{Ser:2020}, provided one can check~\cite[Assumptions 2.1--2.3]{Ser:2020}.

  Our Assumption~\ref{asm:dyn-sep} yields~\cite[Assumption 2.1]{Ser:2020}, and then Assumption~\ref{asm:dyn-sep} item (\ref{itm:reg-var}) together with \cite[Lemma 2.4]{Ser:2020} gives~\cite[Assumption 2.2]{Ser:2020}.
  Finally, as \( F \) is topologically mixing it is exponentially continued fraction mixing (see for example \cite[Section~4]{Aar:1997}) which yields \cite[Assumption 2.3]{Ser:2020} and completes the proof.
\end{proof}

\begin{lemma}\label{lem:k0}
  There is $k_0 \in\bN$ such that for every $i,j\in\{1,\ldots,L\}$,
  there exists $a_{ij}\subset Y_i$ such that $a_{ij}\in\cQ^{k_0}$ and $F^{k_0}(a_{ij})=Y_j$.
\end{lemma}

\begin{proof}
  Fix $i,j \in\{ 1, \ldots , L \}$ and fix $a_i \in \cQ$ with $a_i \subset Y_i$.
  Let $\widetilde Y_j$ denote the part of $Y_j$ which is not contained in any other element of $\cY$, that is,
  $\widetilde Y_j$ is the union of partition elements $a \in \cQ \cap Y_{j}$ so that $\intr (a) \cap Y_{\ell} = \emptyset$ for all $\ell \neq j$.
  As $F$ is topologically mixing we can choose $k_{ij} \in \N$ to be such that,
  for all $n \ge k_{ij}$,
  \[
    F^n(a_i)\cap \widetilde Y_j \ne \emptyset,
  \]
  Let $k_0 \eqdef \max_{1 \le i,j \le L } k_{ij}$. Hence, there is $a_{ij} \in \cQ^{k_0} \cap a_i$ such that $F^{k_0}( a_{ij}) \cap \widetilde Y_j \neq \emptyset$. On the other hand, we know that $F^{k_0}( a_{ij}) \in \{ Y_1, \ldots , Y_L \}$, Hence,  we get $F^{k_0}(a_{ij}) = Y_j$.
\end{proof}

\begin{proof}[Proof of Proposition~\ref{prop:const-An}]
  Given $n \in \N$, let
  \[
    \cB (n) \eqdef
    \{ a \in \cQ^n \colon \vrt_{ n } (a) / \tau_{ n } (a) \in
    B_{ \eps/2 } ( \p ) \}
    ,
  \]
  and let  $\cB(n)^\prime \subset \cB(n)$ be a finite sub-collection so that $\Leb ( \cB(n)^\prime ) \ge \frac{1}{2} \Leb( \cB(n) )$.
  Thus, by Lemma~\ref{lem:uniform-lower-bound}, there exists some constant $C^\prime > 0$ so that
  \begin{equation}\label{eq:Bnprime}
    \Leb \cB(n)^\prime \ge C^\prime,\quad
    \text{ for all $n$ sufficiently large.}
  \end{equation}
  
  Let $k_0$ and $\{ a_{ij} \}_{i,j = 1}^{L}$ be as given by Lemma~\ref{lem:k0}.
  As $F$ is assumed to be a topologically mixing Gibbs-Markov map, each element of $\cY$ is a union of elements of $\cQ$ and the union of the all elements in $\cY$ equals $Y$.
  So, for any cylinder $b \in \cQ^n$ we know from~\eqref{eq:def-im-part}that there exist $p(b),s(b) \in \{ 1, \ldots, L\}$ so that $b \subset Y_{p(b)}$ and $F^n (b) = Y_{s(b)}$. Thus,  for all $i,j\in \{1,\ldots, L\}$ the cylinder $a_{i p(b)} b a_{s(b) j} \in \cQ^{n + 2 k_0}$ is always non-empty and satisfies
  \[
  a_{i p(b)} b a_{s(b)j} \subset Y_{i}, \quand 
  F^{n+2k_0}(a_{ip(b)}ba_{s(b)j}) = Y_j.
  \]
  
  To choose now our collection, given $n > 2k _0$ we set
  \[
    \cA (n) \eqdef \{ a_{i p(b)} b a_{s(b)j} \colon b \in \cB(n-2k_0)',\; i,j = 1,\ldots, L \}
    .
  \]
  By construction, $\cA(n)$ satisfies item (\ref{item3lem:def-An}) of the proposition.

  Letting $i = 1, \ldots, L$ and using Remark~\ref{rem:dist-exp},
  \begin{align*}
    \Leb ( \cA(n) \cap Y_i )
    &= \sum_{j  = 1 }^L \sum_{ b \in \cB(n - 2k_0)^\prime}
    |a_{i p(b)} b a_{s(b)j} | \\
    &\ge D^{-2} \sum_{j  = 1 }^L \sum_{ b \in \cB(n - 2k_0)^\prime}
    |a_{ip(b)}| |b| | a_{s(b)j} | \\
    &\ge
    D^{-2} \left(\min_{\ell,j = 1 ,\ldots,  L } |a_{\ell,j}|\right)^2 \Leb ( \cB( n - 2 k_0)^\prime).
  \end{align*}
  So, for all $n$ large enough that~\eqref{eq:Bnprime} holds there is a $C > 0$ such that $\Leb(\cA(n)) > C$. This concludes item (\ref{item2lem:def-An-new}).

  Now, consider an arbitrary element $a_{ip} b a_{sj} \in \cA(n)$ for some $n \ge 1$.
  From Lemma~\ref{lemlinalgebra} item (\ref{eq:n+k-with-k}) and the fact that $\tau_n \ge n$, we know that 
  \begin{equation}\label{eq:an-subset-1}
    \left|
    \frac{ \vrt_{ n } }{ \tau_{ n } }  (a_{ip} b a_{sj})-
    \frac{ \vrt_{ n - k_0 }  }{ \tau_{ n - k _0 } }(  b a_{sj} )
    \right|
    < 2
    \frac{\tau_{k_0}(a_{ip})}{\tau_{ n - k_0 }( b a_{sj})+\tau_{k_0}(a_{ip})}
    \le \frac{ 2 M }{n}
  \end{equation}
  where
  \[
  	  M \eqdef \max_{ \ell,m = 1 ,\ldots, L } \tau_{k_0}  (a_{ \ell m}).
  \]
  Similarly, by Lemma~\ref{lemlinalgebra}  item~\eqref{eq:n+m-with-n},
  \begin{equation}\label{eq:an-subset-2}
    \left|
    \frac{ \vrt_{ n  - k _0} }{ \tau_{ n - k_0} }  (b a_{sj} )-
    \frac{ \vrt_{ n - 2 k_0 }  }{ \tau_{ n - 2k_0 } }( b ) 
    \right|
    < 2\frac{\tau_{k_0}(a_{sj})}{\tau_{n - 2k_0}(b)+\tau_{k_0}(a_{sj})}
    \le \frac{ 2 M }{n - k_0}.
  \end{equation}
  So, as $b \in \cB(n - 2 k_0)^\prime\subset\cB(n)$ we have
  \begin{equation}\label{eq:b-in-Bn}
    ( \vrt_n / \tau_n ) (b)  \in B_{\eps/2} (\p).
  \end{equation}
  Thus, combining~\eqref{eq:an-subset-1},~\eqref{eq:an-subset-2} and~\eqref{eq:b-in-Bn}, we find that 
  \[
    \cA(n) \subset \{ a \in \cQ^n \colon (\vrt_n/\tau_n) (a) \in B_{3\eps/4} (\p) \}   
  \]
  provided that $n$ is sufficiently large. This yields item (\ref{itm:1-def-An}) and completes the proof.
\end{proof}

\subsection{Return time on $\cA(n)$}
\label{sec:rettime-An}

Let $n$ be large enough so that the conclusions of Proposition~\ref{prop:const-An} hold. Let $\cA(n)$ be as provided by Proposition~\ref{prop:const-An} and let $\Lambda(n) \eqdef \Lambda ( \cA (n) )$.

\begin{lemma}\label{lem:large-dev-tau}
  There exists $N_0 = N_0 (\cA(n)) \in \bN$ such that
  \begin{equation}\label{eq:large-dev-tau}
  | (\vrt_{ \ell }/ \tau_{ \ell }) (a)  - \p \,|<\eps
  \quad
  \text{for every $\ell > N_0$ and every $a \in \cQ^{\ell}_{\Lambda(n)}$}.
  \end{equation}
\end{lemma}

\begin{proof}
We first claim that 
\begin{equation}\label{eq:large-dev-const}
    	\Big|\frac{\vrt_{  kn  }}{ \tau_{  kn  }} (a) - \p\Big| 
    	< \frac{3\eps}{4},
      \quad\text{for every $a\in\cA(n)$},
      \text{ and every } k \ge 1.
\end{equation}
Notice that item (\ref{itm:1-def-An}) of Proposition~\ref{prop:const-An} implies~\eqref{eq:large-dev-const} for $k=1$.
Proceeding inductively on $k$, let $a \in \cQ^{(k+1)n}_{\Lambda(n)}$.
By Remark~\ref{rem:ell-cover-lambda}, we have $a = a_1 \cdots a_ka_{k+1}$, for $a_i \in \cA(n)$. It follows from Lemma~\ref{lemlinalgebra} item~\eqref{eq:n+k-with-min-max} that
  \[
  \frac{ \vrt_{(k + 1)n} } { \tau_{ ( k + 1 ) n } }(a)
  = \frac{ \vrt_{kn+n} } { \tau_{ kn+ n } }(a) 
  \le \max\Big\{ \frac{\vrt_{k n} (a_1 \cdots a_{k })}{ \tau_{ k n } ( a_1 \cdots a_{k} ) },
  			\frac{ \vrt_n (a_{k + 1 })}{\tau_n (a_{k + 1} )} \Big\}
 \]
 together with the analogous lower bound.
 Hence, from \eqref{eq:large-dev-const} and our inductive hypothesis
 \[
 	 \frac{ \vrt_{(k + 1)n} } { \tau_{ ( k + 1 ) n } }(a)
	  \in B_{  3 \eps /  4 }(\p\,),
 \]
  yielding~\eqref{eq:large-dev-const}.

We finally conclude~\eqref{eq:large-dev-tau}. As $\cA(n)$ only contains finitely many cylinders, the value 
\[M \eqdef \max_{ a \in \cA (n)  } \tau_{ n } (a)  \] is finite. Let $b \in \cQ^{ s } \cap \cA(n)$ with $s \in\{1,\ldots, n\}$. It follows from Lemma~\ref{lemlinalgebra}  item~\eqref{eq:n+m-with-n} that
  \begin{equation}
    \label{eq:large-dev-tau-bound}
    \left|
      \frac{ \vrt_{ k n + s } }{ \tau_{ k n + s } } (ab)
      -\frac{\vrt_{ k n   }}{ \tau_{  k n   }} (a)
    \right|
    \le \frac{2\vrt_s(b)}{ \vrt_{ k n } (a) + \vrt_s(b) }
    \le 2 \frac{M}{ k n + s }.
  \end{equation}
  Choose $N_0 \eqdef N_0 (\cA(n)) > n$ so that $2 M / N_0 < \eps / 4$.
 For every $\ell > N_0$ and $a \in \cQ^\ell_{\Lambda(n)}$ relation \eqref{eq:large-dev-tau-bound} together with \eqref{eq:large-dev-const} implies  that $( \vrt_{ \ell } / \tau_{ \ell } )(a) \in B_{ \eps } ( \p )$. This proves the lemma.
\end{proof}

\subsection{Geometric properties of $\cA(n)$}
\label{sec:geom-An}

For each $n$ large enough so that Proposition~\ref{prop:const-An} holds, let $\cA(n)$ be as provided by Proposition~\ref{prop:const-An} and let $\Lambda(n) = \Lambda( \cA(n) )$.

\begin{lemma}\label{lem:dim-An}
  There exists a $C'> 0$ such that $\dim_{\rm H} \Lambda (n)  > 1 - C'/n$ for all $n$ sufficiently large.
\end{lemma}

\begin{proof}
Consider the function $h \colon [0,1] \to \R$ given by
\[
    	h (t) \eqdef \sum_{ a \in \cA(n) } |a|^{ 1 - t }.
\]
Let $s(n) \eqdef 1 -  \vdim \Lambda(n) \in [0,1]$ so that, by~\eqref{defvDim}, $h ( s(n) ) = 1$. It follows from item \eqref{item2lem:def-An-new} in Proposition~\ref{prop:const-An} that
\begin{equation}\label{eqhere}
	h (0) = \sum_{ a \in \cA(n) } |a|
		= \Leb ( \cA(n) ) 
		> C,
\end{equation}
and this property holds for all $n$ sufficiently large. From~\eqref{smallLeb} we get
\begin{align*}
    	h ' (t) 
	= - \sum_{ a \in \cA(n) } \log |a| \cdot |a|^{ 1 - t }
    	&\ge n \log \lambda \cdot \sum_{ a \in \cA (n) } |a|^{ 1 - t } \\
    	&\ge n \log \lambda \cdot \sum_{ a \in \cA (n) } |a|\\
    	\text{by \eqref{eqhere}}\quad
	&\ge C n  \log \lambda
\end{align*}
  giving that $h'(t) \ge Cn \log \lambda$. It follows that
  \[
    s (n)\le \frac{ 1 - h (0) }{ \| h' \|_{ \infty } } \le \frac{ 1 - C }{ C n \log \lambda }
    ,
  \]
  and in particular $\vdim \Lambda (n) = 1 - s(n) +  o(1)$. Using Lemma~\ref{lem:approxvdim}, we get that for some $C' > 0$,
  \[
    \dim_{ \rm H } \Lambda (n)
    \ge \vdim \Lambda (n) - \frac{ D }{ n \log \lambda - D }
    \ge  1 - \frac{ C' }{ n },
  \]
  which yields the assertion of the lemma.
\end{proof}

Fix $m \coloneqq m(n)$ to be the geometric measure for $F^n|_{\Lambda(n)}$ whose existence is guaranteed by Lemma~\ref{existence-m}.

\begin{lemma}\label{lem:uni-geom-An}
  For all $n$ sufficiently large, there exists an $N_1 = N_1 (\cA(n))$ so that
  \[
    \left|
    \frac{ \log m (a) }{ \log |a| } - 1
    \right|
    < \eps
    \quad
    \text{for all $\ell > N_1$ and all $a\in \cQ_{\Lambda(n)}^\ell$}
    .
  \]
\end{lemma}

\begin{proof}
As $m$ is geometric for $F^n|_{\Lambda(n)}$, by~\eqref{eq:def-geom}  there exists an $E(n) > 0$ so that 
  \begin{equation}\label{eq:dim-lamabda-n}
    \dim_{ \rm H } \Lambda (n) + \frac{ E(n) }{ \log |a| }
    \le
    \frac{ \log m (a) }{ \log |a| }
    \le
    \dim_{ \rm H } \Lambda (n) - \frac{ E(n) }{ \log |a| }
    .
  \end{equation}
 Using~\eqref{smallLeb},
 we can choose $N_1 = N_1 (\cA(n)) \in\bN$ so that for $\ell \ge N_1$,
 \begin{equation}\label{eq:error-dim-lambda-n}
 	\sup_{a\in\cQ^\ell}\frac{E(n)}{-\log|a|}
	\le \frac{E(n)}{\ell \log\lambda}
	<\eps/2
  .
 \end{equation}
 If $n > 2C^\prime / \eps$, then Lemma~\ref{lem:dim-An} ensures $\dim_{ \rm H } \Lambda (n) \ge 1 - \eps/2$. Combining this with~\eqref{eq:dim-lamabda-n} and~\eqref{eq:error-dim-lambda-n} concludes the proof.
\end{proof}

\subsection{Proof of Proposition~\ref{prop:approx-new}}
\label{sec:proof-of-approx}

Fix $n$ large enough so that Lemma~\ref{lem:uni-geom-An} and  Proposition~\ref{prop:const-An} hold.
Let $\cA = \cA(n) \subset \cQ^n$  be as in Proposition~\ref{prop:const-An} and let $N_0,N_1$ be as in Lemmas~\ref{lem:large-dev-tau} and~\ref{lem:uni-geom-An}. Fix $N \coloneqq \max\{ N_0, N_1\}$ and let $m$ be the geometric measure for $F^n|_{\Lambda ( \cA)}$ whose existence is guaranteed by Lemma~\ref{existence-m}. Then, Lemmas~\ref{lem:large-dev-tau} and~\ref{lem:uni-geom-An} yield~\eqref{eq:approx}. Moreover, $m ( \Lambda ( \cA ) )= 1$ and $F^n ( \cA ) = \cY$ by item (\ref{item3lem:def-An}) of Proposition~\ref{prop:const-An}.

Fix $i \in \{ 1 , \ldots, L\}$. Given $r > 0$, let $\cB(r)$ be a finite collection of balls of diameter $r$ so that  $\cB (r) \subset \cA \cap Y_i$  and that $\Leb (\cB(r) ) \ge \frac12\Leb ( \cA \cap Y_i )$. Then, from~\eqref{eq:def-geom}, we know that  for $r > 0$ sufficiently small there is a $E' > 0$ 
\begin{align*}
   m( \Lambda ( \cA ) \cap Y_i ) 
   = m ( \cA \cap Y_i )
   \ge \sum_{ b \in \cB(r) } m (b) 
   \ge E' \sum_{ b \in \cB(r) } |b|^d
   &\ge E' r^{d-1} \Leb ( \cB(r) ) \\
   &\ge E' r^{d-1} \frac12 \Leb (  \cA \cap Y_i  ).
\end{align*}
The latter is strictly positive by item (\ref{item2lem:def-An-new}) of Proposition~\ref{prop:const-An}.
This concludes the proof.
\qed

\section{Proof of Corollary~\ref{cor:dim-H-nup}}
\label{sec:bridging}

Throughout this section we fix an arbitrary $\p \in \cS$. We are going to show that
\begin{equation}\label{eq:dim-H-gp}
  \dim_{\rm H} ( \mathscr{G} (\nu_\p) )  = 1.
\end{equation}
To do so, we first construct a fractal set $\Gamma$, see Section~\ref{sec:const-gamma}. In Section \ref{sec:const-m} we construct a Borel probability measure supported on $\Gamma$. Section \ref{sec:choice-of-ki} states some auxiliary results.  In Section~\ref{sec:subset} we show that $\Gamma \subset \mathscr{G} (\nu_\p)$. In Section~\ref{sec:dim} we prove $\dim_{\rm H} ( \Gamma ) = 1$ which implies~\eqref{eq:dim-H-gp}.

\subsection{Construction of a fractal set $\Gamma$}\label{sec:const-gamma}

Fix a sequence $(\eps_{ i })_{i\ge 0} \subset (0,1)$ such that $\lim_{i\to\infty}\eps_{ i } = 0$.
For each $i \ge 0$ fix
\[
	n_i = n(\varepsilon_i, \p), \;\;
	\cA_{ i } = \cA (\eps_{ i }, \p)\subset\cQ^{n_i}, 
	\;\; \Lambda_{ i } = \Lambda ( \cA_{ i } ), \;\;
	N_{ i } = N (\eps_{ i }, \p) ,
  \; \text{ and } \;
  m_{ i } = m( \eps_{ i }, \p)
\]	 
as given by Proposition~\ref{prop:approx-new}.

Let $(k_{ i })_{i \ge 0} \subset \N$ be an arbitrary sequence of integers.
In the remainder of this subsection, we construct a fractal set $\Gamma = \Gamma ( (k_i)_i )$ which will depend on this sequence. The precise choice of $(k_i)_{i}$, and hence the precise choice of $\Gamma$, will  be fixed in Section~\ref{sec:choice-of-ki}.
Set
\[
  t_{ 0 } \eqdef 0,
  \quand
  t_{ i } \eqdef
  n_0 k_0+\cdots+n_{i-1}k_{i-1}
  ,
\]
and define
\begin{equation}\label{eqnested}
  \cB_{ i } \eqdef \bigvee_{ j = 0 }^{ i - 1 } F^{ -t_{ j } } \cQ_{\Lambda_{ j }}^{ k_j n_j }
  \quand
  \Gamma
  =
  \Gamma\big((k_i)_i\big)
  \eqdef
  \bigcap_{ i \in\bN }
  \bigcup_{ b \in \cB_{i} } b 
  .
\end{equation}
Heuristically, the set $\Gamma$ consists of points that enter a neighbourhood of $\Lambda_i$ at time $t_i$ and spend time $n_ik_i$ close to $\Lambda_i$ before moving close to $\Lambda_{i+1}$ at time $t_{i+1}$.
Note that for each $i\ge 0$, $\cB_{ i } \subset \cQ^{ t_{ i } }$ and that the collection $\cB_{ i }$ covers $\Gamma$. By definition, the collections $(\cB_i)_i$ are ``nested'' in the sense that every element in $\cB_{i+1}$ is a subset of an element of $\cB_i$, for every $i$.

\begin{lemma}\label{lem:nonempty-compact}
	$\Gamma$ is nonempty and compact.
\end{lemma}

\begin{proof}
	It is enough to check that for every index $i\ge 0$ for every $a\in\cA_{i+1}$ there exists $a'\in\cA_i$ such that $F^{n_i}(a')\supset a$. This follows from the fact that  Proposition \ref{prop:approx-new} ensures  $F^{n_i} ( \cA_i ) = \cY$. So, every $\cB_i$ is nonempty and the intersection in \eqref{eqnested} is the intersection of a nested sequence of compact sets. This implies the assertion.
\end{proof}

\subsection{Construction of $m$}\label{sec:const-m}

We now  construct a Borel measure $m$ on $Y$ satisfying $m ( \Gamma ) =1$. We first define $m$ on cylinders $b a \in \cQ^{t_i + s}$ where $b\in \cQ^{t_i}$ and $a \in \cQ^{s}$ for $s = 1,\ldots,n_{i}k_{i}$, inductively on $i$.
For $i = 0$, $s =1,\ldots,n_1 k_1$, and $a \in\cQ^s$, we set
\[
m(a) = m_0 (a).
\]
Proceeding inductively, for $i \ge 1$, $s=1,\ldots,n_{i}k_{i}$, and $b a \in \cQ^{t_i + s}$ with $b \in \cQ^{t_i}$ and $a\in \cQ^s$, we define
\begin{equation}\label{eq:def-m-new}
  m(ba) = C_i(b) m(b) m_i (a)
\end{equation}
where
\begin{equation}
  C_{i} (b) \eqdef
  \left(
    m_i (
    \Lambda_i \cap
    F^{t_i} (b))
  \right)^{-1}
  =
    \left(
    m_i (
    F^{t_i} (b))
  \right)^{-1}.
\end{equation}

\begin{remark} \label{rem:Ci-positive}
  As $F^{t_i} (b) \in \cY$, Proposition~\ref{prop:approx-new} ensures that $C_i (b) > 0$ for all $b \in \cB_i$.
\end{remark}

This defines $m$ on all of $\cQ^*$. One readily checks that this extends to a Borel probability measure on $Y$. By construction $m(\cB_i) = 1$ for every $i\ge 1$, and in particular $m(\Gamma)=1$.

\subsection{Choice of the sequence $(k_i)_i$}\label{sec:choice-of-ki}

In the section we fix the choice of the sequence $(k_i)_i$. First, we introduce
\begin{equation}
  \label{eq:def-M}
  M_{ i } 
  \eqdef
  \max_{ a \in \cA_i }
  \max_{ 0 \le s \le n_i - 1 }
  \tau (F^{s}(a))
  ,
\end{equation}
and
\begin{equation}\label{eq:def-Ci}
  \widetilde C_i \coloneqq \sup_{ b \in \cQ_{\Gamma}^{t_i} } |\log C_i (b) |
  = \max_{j=1,\ldots,L }| \log m_i ( Y_j ) |
  .
\end{equation}

For every $x\in \Gamma$ and each $i\ge 0$ and $s = 1,\ldots, n_{i}k_{i}$ define the cylinders
\begin{equation}\label{eq:def-ai-bi-ais-1}
  a_{i,s} (x)\in \cQ_{\Lambda_i}^s
  \quad
  a_{i} (x) \in \cQ_{\Lambda_i}^{n_{i}k_{i}},
  \quand b_i(x) \in\cB_i
\end{equation}
such that
\begin{equation}\label{eq:def-ai-bi-ais-2}
  x \in b_{i}(x)a_{i,s} (x), \quad
  a_{i}(x) = a_{i,n_{i}k_i} (x),\quand
  b_i (x) = a_{0} (x) \cdots a_{i - 1} (x).
\end{equation}
In what is below, we will regularly suppress the dependence of the associated $b_i$ and $a_i$ on $x$ to ease notation.
See Figure~\ref{fig:Gamma}.
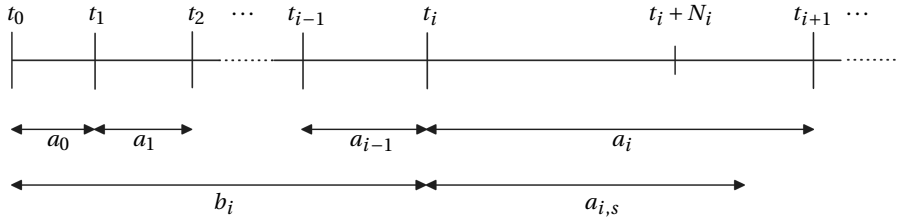
\begin{figure}[h] 
 \begin{overpic}[scale=.13]{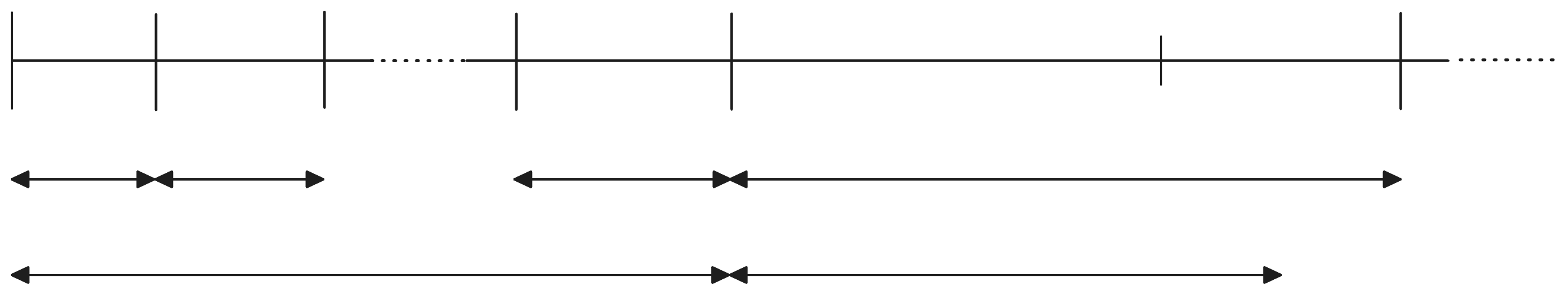}
 	\put(0,19){\small{$t_0$}}
 	\put(4.5,5){\small{$a_0$}}
 	\put(9,19){\small{$t_1$}}
 	\put(14,5){\small{$a_1$}}
 	\put(20,19){\small{$t_2$}}
 	\put(25,19){\small{$\cdots$}}
 	\put(31,19){\small{$t_{i-1}$}}
 	\put(38,5){\small{$a_{i-1}$}}
 	\put(46,19){\small{$t_i$}}
  	\put(67,5){\small{$a_i$}}
	\put(71,19){\small{$t_i+N_i$}}
 	\put(87,19){\small{$t_{i+1}$}}
 	\put(93,19){\small{$\cdots$}}
 	\put(23,-2){\small{$b_i$}}
  	\put(64,-2){\small{$a_{i,s}$}}
\end{overpic}
  \label{fig:Gamma}
  \caption{A sketch showing cylinders $b_i,a_{i}$ and $a_{i,s}$ as defined in~\eqref{eq:def-ai-bi-ais-1} and~\eqref{eq:def-ai-bi-ais-2}.}
\end{figure}

We can now state our precise conditions on $(k_i)_i$.
\begin{align}
  \label{eq:k_i-1}
  & k_i > \frac{ N_i }{ n_i } \\
  \label{eq:k_i-2}
  & \frac{ N_{i + 1 } M_{ i + 1 } }{ t_{i+1} } < \eps_i \\
  \label{eq:k_i-3}
  & \frac{ \sum_{ j = 0 }^{i - 1} M_j n_j k_j }{ n_{i} k_{i}} < \eps_i \\
  \label{eq:k-4-new}
  &\sup_{x \in \Gamma}
  \frac{ |\log m (b_{i} (x)) | + \widetilde C_{i} }{ |\log m_i (a_i (x) )| }
  \le \eps_i, \\
  \label{eq:k-5-new}
  &\sup_{ x \in \Gamma}
   \frac{ \big|\log |b_i (x) |\big| +  |\log D| }{\big|\log |a_i(x)|\big|}
  \le \eps_i
  , \\
  \label{eq:k-6-new}
  &\sup_{x \in \Gamma}
  \sup_{1 \le s \le N_i}
  \frac{|\log m_i( a_{i,s} (x) )| + \widetilde C_{i} }{  | \log m (b_i) |  }
  <\eps_{i-1}
  , \\
  \label{eq:k-7-new}
  &\sup_{ x \in \Gamma }
  \sup_{1 \le s \le N_i}
  \frac{\big|\log |a_{i,s}(x)|\big| + |\log D| }{ \big| \log |b_i(x)|\big| } < \eps_{i - 1}.
\end{align}

The remainder of the section we prove the following proposition.

\begin{proposition}\label{prop:ki}
  There exists a sequence $(k_i)_i$ such that~\eqref{eq:k_i-1}--\eqref{eq:k-7-new} hold.
\end{proposition}

We divide the proof into the following lemmas. Recall from Proposition~\ref{prop:approx-new} that
\begin{equation} \label{eq:geom-meas-uni}
  \frac{ \log m_i (a_{i,s} (x)) }{ \log |a_{i,s} (x)| } \in B_{\eps_i} ( 1 )
  \quad \text{for all } N_i \le s \le k_i n_i
  .
\end{equation}

\begin{lemma}\label{lem:bi-1}
  If~\eqref{eq:k-4-new} and~\eqref{eq:k-5-new} holds, then for every $x \in \Gamma$,
  \[
    \lim_{ i \to \infty }
    \frac{ \log m ( b_{ i } (x) ) }{ \log | b_{ i} (x) | }
    = 1
    .
  \]
\end{lemma}

\begin{proof}
  Let $x\in \Gamma$ and consider the associated sequence $(a_i)_i$ and $(b_i)_i$ in~\eqref{eq:def-ai-bi-ais-1} and~\eqref{eq:def-ai-bi-ais-2}.
  Using the definition of $m$ in~\eqref{eq:def-m-new}, the bounded distortion of $F$ (see Remark~\ref{rem:dist-exp}) we get 
  \begin{align*}
    \frac{ \log m ( b_{ i + 1} ) }{ \log |b_{ i + 1}| }
    \ge
    \frac{ \log m ( b_{ i } a_i ) }{ \log  |b_{ i } a_i| }
    &\ge
    \frac{ \log m ( b_i ) + \log m_i (a_i) + \log C_i (b_i) }
    { \log |b_i|  + \log |a_i| - \log D } \\
    &\ge
    \frac{ \log m_i (a_i) }{ \log |a_i| }
    \left(
      \frac{ 1
       + \frac{ \log m ( b_i ) + \widetilde{C}_i }{ \log m_i (a_i) }
      }{ 1
       + \frac{ \log |b_i| - \log D }{ \log |a_i|}
      }
    \right)
    .
  \end{align*}
  Using~\eqref{eq:geom-meas-uni}, together with~\eqref{eq:k-4-new} and~\eqref{eq:k-5-new}, we find that
  \[
    \frac{ \log m ( b_{ i + 1} ) }{ \log  |b_{ i + 1}| } \ge 
    (1 - \eps_i) \frac{ 1 - \eps_i }{ 1 + \eps_i }.
  \]
  Proceeding in the same way, one obtains the upper bound
  \[
    \frac{ \log m ( b_{i + 1} )}{\log |b_{i + 1 }| }
    \le
    (1 + \eps_i) \frac{ 1 + \eps_i }{ 1 - \eps_i }.
  \]
  Taking the limit as $i\to \infty$ concludes the proof.
\end{proof}

\begin{lemma}\label{lem:ki-4-and-5}
  If the sequence $(k_i)_{i}$ grows sufficiently quickly,
  then~\eqref{eq:k-4-new} and~\eqref{eq:k-5-new} hold.
\end{lemma}

\begin{proof}
  Note that the collection $\{ b_i (x) \colon x \in \Gamma \}$ is finite. We also know that the quantities on the tops of the fractions in~\eqref{eq:k-4-new} and~\eqref{eq:k-5-new} depend on $k_0,\ldots,k_{i-1}$ but not $k_i$. So, there exists some constant $\gamma_i < +\infty$ which is independent of $k_i$ (but may depend on $k_0,\ldots,k_{i-1}$) such that
  \[
    \sup_{x\in\Gamma} |\log m ( b_i (x) )| + \widetilde C_i
    < \gamma_i,
    \quand
    \sup_{x\in\Gamma}\big| \log |b_i (x) |\big| + |\log D | < \gamma_i.
  \]
  
  We know from Remark~\ref{eq:distotion-expansion} that $\big|\log |a_i(x)|\big| \ge  k_{i} n_i \log \lambda$.
  We also know from~\eqref{eq:geom-meas-uni} that $\log m_{i} (a_i) / \log |a_{i} | \in B_{\eps_i} (1)$. In particular, we must have that
  \[  
    | \log m_{i} (a_i)|
    \ge (1 - \eps_i ) n_i k_i \log \lambda.
  \]

  So, we can choose $k_i$ recursively to grow fast enough so that~\eqref{eq:k-4-new} and~\eqref{eq:k-5-new} hold.
\end{proof}

\begin{lemma}\label{lem:ki-6-and-7}
  If the sequence $(k_i)_{i}$ grows sufficiently quickly,
  then~\eqref{eq:k-6-new} and~\eqref{eq:k-7-new} hold.
\end{lemma}

\begin{proof}
  Suppose that $(k_i)_i$ is chosen to grow fast enough so that~\eqref{eq:k-4-new} and~\eqref{eq:k-5-new} hold by Lemma~\ref{lem:ki-4-and-5}.
  The top of the fractions in~\eqref{eq:k-6-new} and~\eqref{eq:k-7-new} do not depend on the choice of sequence $(k_i)_i$. On the other hand we know that
  \[
    \big| \log |b_i (x)|\big|
    \ge n_{ i - 1 }k_{i - 1 }\log \lambda.
  \]
  Moreover, Lemma~\ref{lem:bi-1} tells us that then
  \[
    \big| \log |b_i (x)|\big|
    \ge ( 1 - \gamma \eps_{ i - 1 } )
    n_{ i - 1 }k_{i - 1 }\log \lambda
    .
  \]
  So, choosing $(k_i)_{i}$ to grow sufficiently quickly, we can ensure that~\eqref{eq:k-6-new} and~\eqref{eq:k-7-new} hold.
\end{proof}

\begin{proof}
  [Proof of Proposition~\ref{prop:ki}]
  Conditions~\eqref{eq:k_i-1} and~\eqref{eq:k_i-2} are satisfied whenever $k_i$ is chosen to be large enough. Moreover, the top of the fraction in~\eqref{eq:k_i-3} depends only on $k_0,\ldots,k_{i-1}$ and not on $k_i$. These observations together with Lemmas~\ref{lem:ki-4-and-5} and~\ref{lem:ki-6-and-7} ensure that $(k_i)_i$ can be chosen such that~\eqref{eq:k_i-1}--~\eqref{eq:k-7-new} hold. This completes the proof.
\end{proof}

For the remainder of this section we fix $(k_i)_i$ satisfying~\eqref{eq:k_i-1}--\eqref{eq:k-7-new} and fix the corresponding choices of $\cB_i$, $\Gamma$, and $m$.

\subsection{Limit behaviour of points in $\Gamma$}
\label{sec:subset}

\begin{lemma}\label{lem:Vx-np}
  For all $x\in \Gamma$, $\cV(x) = \nu_{ \p}$.
\end{lemma}

\begin{proof}
Let $x \in \Gamma$.
By Corollary~\ref{cor:coding-single-point}, it is enough to show that $\cT(x) = \p$.

Recall the definition of the associated cylinders $a_{i,s} \in \cQ_{\Lambda_i}^{s}$, $a_{i} \in \cQ_{\Lambda_i}^{n_ik_i}$ and $b_i \in \cB_i$ in~\eqref{eq:def-ai-bi-ais-1} and~\eqref{eq:def-ai-bi-ais-2}.
Recall from Proposition~\ref{prop:approx-new} that
\begin{equation}\label{eq:large-dev-tau-1}
  | ( \vrt_{ s }   / \tau_{ s } )(a_{i,s}) - \p |<\eps_i
   \quad\text{ for all } N_i\le s \le n_ik_i.
\end{equation}

We also make use of the following inequalities which come from the definition~\eqref{eq:def-M} and the fact that $\tau \ge 1$
\begin{equation}
  \label{eq:tau-trivial}
  t_i \le \tau_{t_i} ( b_i ) \le \sum_{ j = 0 }^{ i - 1 }
   M_j n_j k_j,
  \quand
  s  \le \tau_s ( a_{ i, s} )
  \le s M_i
  ,
\end{equation}
for all $s = 1,\ldots, n_i k_i$.

  Recalling that $b_i = a_0 \cdots a_{i - 1}$,
  we can write $b_{ i + 1 } = b_i a_i$. Then, from 
  Lemma~\ref{lemlinalgebra} item~\eqref{eq:n+k-with-k} together with~\eqref{eq:tau-trivial}
  and~\eqref{eq:k_i-3},
  \begin{align*}
    \left|
    \frac{ \vrt_{t_{i + 1}} }{ \tau_{ t_{ i + 1 } } } (b_{ i + 1 }) -
    \frac{ \vrt_{ n_{i }k_{i}  } }{\tau_{ n_{i }k_{i} } } (a_{i} )
    \right|
    \le \frac{ 2 \tau_{t_i}(b_i)}{\tau_{t_i}(b_i) + \tau_{n_{i}k_{i}} ( a_{i} )}
    &< \frac{ 2\sum_{ j = 0 }^{ i - 1 }  M_{ j }n_{ j }k_{ j }  }{ n_{ i } k_{ i } } 
    < 2 \eps_{ i }
    .
  \end{align*}
  So, using~\eqref{eq:large-dev-tau-1} with $s = n_{i}k_{ i }$, we obtain
  \begin{equation}
    \label{eq:good-at-somepoint}
    (\vrt_{ t_{ i + 1 } } / \tau_{ t_{ i + 1 } } ) (b_{ i + 1 }) \in B_{3\eps_{ i }}(\p),
  \quad \text{for all } i \ge 0
  .
  \end{equation}

  We now fix $i \ge 1$ and consider the intermediate times $t_{i} + s$ where $s= 1 ,\ldots, n_{ i } k_i - 1$.
  Let us first consider the range $s = 1 , \ldots , N_{i} - 1$. By Lemma~\ref{lemlinalgebra} item~\eqref{eq:n+m-with-n} together with~\eqref{eq:tau-trivial} and~\eqref{eq:k_i-2} we have that
  \begin{align*}
    \left|
    \frac{ \vrt_{  t_{ i }  + s }  }{ \tau_{  t_{ i } + s } }( b_i a_{i, s} ) 
    -
    \frac{ \vrt_{ t_{ i } } }{ \tau_{ t_{ i } }  }(b_i)
    \right|
    \le 2\frac{ \tau_s ( a_{i,s} ) }{ \tau_{t_i}( b_i )+
    \tau_s ( a_{i,s })}
    &\le \frac{  2N_{ i } M_{ i } }{ t_{ i } }
    < 2 \eps_{ i - 1 },
  \end{align*}
  So, using~\eqref{eq:good-at-somepoint}, we know that
  \[
    ( \vrt_{ t_i + s }  / \tau_{ t_i + s } ) ( b_i a_{ i , s } )
    \in B_{ 5 \eps_{ i- 1} } ( \p )
    .
  \]
  Now let us consider the range $s = N_{ i }, \ldots, n_{i} k_{i} - 1$. By Lemma~\ref{lemlinalgebra} item~\eqref{eq:n+k-with-min-max} we know that 
   \[
     \min\left\{
      \frac{ \vrt_{ t_{ i } }}{ \tau_{ t_{ i } }  }(b_i),
      \frac{ \vrt_{ s }  }{ \tau_{ s }  }( a_{ i, s } )
    \right\}
    \le
        \frac{ \vrt_{  t_{ i } +s } }{ \tau_{  t_{ i } +s}  }(b_i a_{i,s})
    \le
    \max\left\{
      \frac{ \vrt_{ t_{ i } } }{ \tau_{ t_{ i } }  }(b_i),
      \frac{ \vrt_{ s }  }{ \tau_{ s }  }( a_{ i, s } )
    \right\}.
 \]
 Thus, by~\eqref{eq:large-dev-tau} and~\eqref{eq:good-at-somepoint} we find that
 \[
   (\vrt_{  t_{ i } +s }/ \tau_{  t_{ i } +s} ) (b_i a_{i,s})  \in B_{ \delta_i } ( \p ),
 \]
where $\delta_i = \max\{ 3\eps_{ i - 1}, \eps_{i} \}$. As we assume that $\eps_i\to0$, this proves $\cT(x) = \p$. This finishes the proof of the lemma.
\end{proof}

\subsection{Hausdorff dimension of $\Gamma$}
\label{sec:dim}

\begin{lemma}\label{lem:dim}
  $
    \dim_{\rm H} \Gamma  = 1
    .
  $
\end{lemma}

\begin{proof}
   Let $x\in \Gamma$. 
  Recall the definition of the associated cylinders $a_{i,s} \in \cQ_{\Lambda_i}^{s}$, $a_{i}\in \cQ_{\Lambda_i}^{n_ik_i}$ and $b_i\in \cB_i$ in~\eqref{eq:def-ai-bi-ais-1} and~\eqref{eq:def-ai-bi-ais-2}.
  We will show that
  \begin{equation}\label{eq:dim-claim}
    \lim_{i\to\infty}
    \inf_{s = 0 ,\ldots, n_i k_i - 1} \frac{ \log m ( b_i a_{i,s} ) }{ \log |b_i a_{i,s} |}
    =
    \lim_{ i \to \infty }\sup_{s = 0 ,\ldots, n_i k_i - 1} \frac{ \log m ( b_i a_{i,s} ) }{ \log |b_i a_{i,s} |}
    = 1.
  \end{equation}

  Set
  \[
    \gamma_i \eqdef \frac{ \log m (b_i)} { \log |b_i|},
  \]
  and recall from Lemma~\ref{lem:bi-1} that $\lim_{i\to\infty} \gamma_i = 1$
  Consider first the range $s = 0, \ldots, N_i$.
  We proceed as in the proof of Lemma~\ref{lem:bi-1}.
  Using the definition~\eqref{eq:def-m-new}, the bounded distortion (see Remark~\ref{rem:dist-exp})
  \begin{align}
    \nonumber
    \frac{
      \log m ( b_i a_{i,s})
    }
    {
      \log |b_i a_{i,s}|
    }
    &\ge 
    \frac{ \log m ( b_i ) + \log m_i (a_{i,s}) + \log C_i (b_i) }
    { \log |b_{i}|  + \log |a_{i,s}| - \log D } \\
    \nonumber
    &\ge
    \frac{ \log m  (b_i) }{ \log |b_i| }
    \left(
      \frac{ 1
       + \frac{ \log m ( a_{i,s} ) + \widetilde C_{i} }{ \log m (b_i) }
      }{ 1
       + \frac{ \log |a_{i,s}| - \log D }{ \log |b_i|}
      }
    \right) \\
    \label{eq:dim-s-small}
  &\ge
  \gamma_i \frac{ 1 - \eps_{ i - 1} }{ 1 + \eps_{ i - 1 } }.
  \end{align}
  where in the final inequality we have
  used~\eqref{eq:k-6-new} and~\eqref{eq:k-7-new}.

  Now let us consider the range $s = N_i, \ldots ,n_{i}k_i$. We begin as before, but this time we make use of~\eqref{eq:abcd-clever} to obtain
  \begin{align}
    \nonumber
    \frac{
      \log m ( b_i a_{i,s})
    }
    {
      \log |b_i a_{i,s}|
    }
    &\ge 
    \frac{ \log m ( b_i ) + \log m_i (a_{i,s}) + \widetilde C_i }
    { \log |b_{i}|  + \log |a_{i,s}| - \log D } \\
    \label{eq:dim-s-big-1}
    &\ge
    \min\left\{
    \frac{ \log m ( b_i ) + \widetilde C_i }
    { \log |b_{i}|  - \log D },
    \frac{ \log m_i (a_{i,s}) }
    {  \log |a_{i,s}| } 
    \right\}
  \end{align}
  Then we notice, using~\eqref{eq:k-6-new} and~\eqref{eq:k-7-new}, that
  \begin{equation}
    \label{eq:dim-s-big-2}
    \frac{ \log m ( b_i ) + \widetilde C_i }
    { \log |b_{i}|  - \log D }
    =
    \frac{ \log m ( b_i )}
    { \log |b_{i}|  }
    \left(
    \frac{ 1 + \frac{\widetilde C_i}{\log m ( b_i )} }
    { 1 + \frac{\log D }{\log |b_{i}|}}
    \right)
    \ge \gamma_i \frac{ 1 - \eps_{ i - 1 } }{ 1 + \eps_{ i - 1 } }.
  \end{equation}
  Recalling~\eqref{eq:geom-meas-uni}, we have $\log m_i (a_{i,s}) / \log |a_{i,s}| \ge 1 - \eps_i$. Combining this~\eqref{eq:dim-s-small},~\eqref{eq:dim-s-big-1}, and~\eqref{eq:dim-s-big-2}, we obtain
  \[
    \inf_{s = 0 ,\ldots, n_i k_i - 1} \frac{ \log m ( b_i a_{i,s} ) }{ \log |b_i a_{i,s}|} \ge \underline{\delta}_i,
    \quad
    \text{where }
    \underline{\delta}_i \eqdef
    \min\left\{
    \gamma_i \frac{ 1 - \eps_{ i - 1 }} {1  + \eps_{i - 1 }},
    1  - \eps_i
    \right\}.
  \]
  Proceeding in the same as above one can find analogous upper bounds for~\eqref{eq:dim-s-small},~\eqref{eq:dim-s-big-1}, and~\eqref{eq:dim-s-big-2}, which yield
  \[
    \sup_{s = 0 ,\ldots, n_i k_i - 1} \frac{ \log m ( b_i a_{i,s} ) }{ \log |b_i a_{i,s}|} \le \overline{\delta}_i,
    \quad
    \text{where }
    \overline{\delta}_i \eqdef
    \max\left\{
    \gamma_i \frac{ 1 + \eps_{ i - 1 }} {1  -  \eps_{i - 1 }},
    1  + \eps_i
    \right\}.
  \]
  As $\lim_{i \to \infty} \underline{\delta}_i = \lim_{ i \to \infty } \overline{\delta}_i = 1$ we obtain~\eqref{eq:dim-claim}.

  Denoting by $\cQ^\ell(x)$ the element of the partition $\cQ^\ell$ which contains $x$, from the above we get
  \[
    \lim_{\ell \to \infty}\frac{ \log m ( \cQ^{\ell}(x) ) }{ \log|\cQ^{\ell}(x)|} = 1.
  \]
  As the refinements of the partition $\cQ$ separate points, the assertion of the lemma then follows from the mass distribution principle.
\end{proof}

\section{Proof of Theorem~\ref{thm:dim-H-C}}
\label{sec:bridging-again}

In this section we prove Theorem~\ref{thm:dim-H-C}. We will follow the same structure as for the proof of the special case in Section~\ref{sec:bridging}. The only significant difference is how the sets $\cA_{i}$ are chosen in the construction of the fractal set $\Gamma$. Throughout this section we fix a closed  connected subset $\cC \subset \cS$. In Section~\ref{sec:const-gamma-2} we construct a set $\Gamma$. In Section~\ref{sec:subset-2} we show that $\Gamma \subset \{ x \colon  \cV(x) = \pi \cC \}$. In Section \ref{sec73}, we conclude $\dim_{\rm H} \Gamma = 1$ which proves Theorem~\ref{thm:dim-H-C}.

\subsection{Construction of a fractal set $\Gamma$}
\label{sec:const-gamma-2}

Let $(\p_i)_{i\ge 0} \subset \cC$ be a sequence whose limit points are equal to $\cC$ and so that $\lim_{i \to \infty }| \p_{i} - \p_{i + 1 } | = 0$.
Let $(\eps_i)_{i\ge 0} \subset (0,1)$ be such that $\lim \eps_i = 0$. For each $i\ge 0$ let $n_i \eqdef n_i (\eps_i,\p_i )$, $\cA_i = \cA( \eps_i ,\p_i) \subset \cQ^{n_i}$, $\Lambda_i = \Lambda( \cA_i )$, $N_i = N( \eps_i,\p_i)$, and $m_i = m ( \eps_i,\p_i)$ be as given in Proposition~\ref{prop:approx-new}.
With these definitions in place we define $M_i,E_i$ as in~\eqref{eq:def-M}. Then, we may repeat verbatim the proof of Proposition~\ref{prop:ki} to give the existence of a sequence $(k_i)_{i \ge 0}$ satisfying the conditions~\eqref{eq:k_i-1}--\eqref{eq:k-7-new}. Fixing such a sequence $(k_i)_{i}$ we set
\[
  t_{ 0 } \eqdef 0,
  \quand
  t_{ i } \eqdef
  n_0 k_0+\cdots+n_{i-1}k_{i-1}
  ,
\]
and define
\[
  \cB_{ i } \eqdef \bigvee_{ j = 0 }^{ i - 1 } F^{ -t_{ j } } \cQ_{\Lambda_{ j }}^{ k_j n_j }
  \quand
  \Gamma \eqdef
  \bigcap_{ i \in\bN }
  \bigcup_{ b \in \cB_{i} } b
  .
\]
Repeating the proof of Lemma~\ref{lem:nonempty-compact} we see that $\Gamma$ is nonempty and compact. Following Section~\ref{sec:const-m} we construct a probability measure $m$ on $\Gamma$ exactly as before. 

\subsection{Limit behaviour of points in $\Gamma$}
\label{sec:subset-2}

\begin{lemma}
  \label{lem:large-dev-tau-again}
  $\cV(x) = \pi\cC$ for every $x\in \Gamma$.
\end{lemma}

\begin{proof}
Let $x \in \Gamma$.
From Theorem~\ref{thm:coding} it is enough to show that $\cT(x) = \cC$ and \eqref{eqcondtaus} holds.

Recall the definition of the cylinders $a_{i,s}(x) \in \cQ_{\Lambda_i}^{s}$, $a_{i}(x) \in \cQ_{\Lambda_i}^{n_ik_i}(x)$ and $b_i(x) \in \cB_i$ in~\eqref{eq:def-ai-bi-ais-1} and~\eqref{eq:def-ai-bi-ais-2} (see also Figure~\ref{fig:Gamma}).
We suppress the dependence of $a_{i,s},a_i$ and $b_i$ on $x$ to ease notation.
Recall from item~(\ref{itm:1-def-An}) of Proposition~\ref{prop:approx-new} that
\begin{equation}\label{eq:large-dev-tau-2}
  | ( \vrt_{ s }   / \tau_{ s } )(a_{i,s}) - \p_i |<\eps_i
   \quad\text{ for all } N_i\le s \le n_ik_i.
\end{equation}

  Writing $b_{i+1} = b_i a_{i}$, it follows from Lemma~\ref{lemlinalgebra} item~\eqref{eq:n+k-with-k} together with~\eqref{eq:tau-trivial}
  and~\eqref{eq:k_i-3} that
  \begin{align*}
    \left|
    \frac{ \vrt_{t_{i+1}} }{ \tau_{ t_{ i+1 } } }
    (b_{i+1})
    -
    \frac{ \vrt_{ n_{i }k_{i}  }  }{\tau_{ n_{i} k_{i}}}
    (a_{i} )
    \right|
    &\le \frac{ 2 \tau_{t_i}(b_{i})}{\tau_{t_i}(b_{i}) + \tau_{n_{i}k_{i}} ( a_{i} )}
    < \frac{ 2\sum_{ j = 0 }^{ i - 1 } k_{ j } n_{ j } M_{ j } }{ n_{ i } k_{ i } }
    < 2 \eps_{ i }
    .
  \end{align*}
  So, using~\eqref{eq:large-dev-tau-2} with $s = n_{i}k_{ i }$, we obtain
  \begin{equation}
    \label{eq:good-at-somepoint-again}
    (\vrt_{ t_{ i + 1 }} / \tau_{ t_{ i + 1} } )(b_{ i + 1 }) \in B_{3\eps_{ i }}(\p_i),
  \quad \text{for all } i \ge 0
  .
  \end{equation}
  This gives the conclusion of the lemma along the subsequence $t_i$, in the sense that
  \[
    \big\{ \text{limit points of } (\vrt_{t_i}(x)/\tau_{t_i}(x))_{i \in\bN}\big\} 
    = \big\{ \text{limit points of } (\p_i)_{i\in\bN} \big\} = \cC.
  \]

  We now fix $i \ge 1$ and consider the intermediate times $t_{i} + s$ where $s= 1 ,\ldots, n_{i}k_{i} - 1$.
  Let us first consider the range $s = 1 , \ldots , N_{i} - 1$. By Lemma~\ref{lemlinalgebra} item~\eqref{eq:n+m-with-n} together with~\eqref{eq:k_i-2} and~\eqref{eq:tau-trivial} we have that
  \begin{align*}
    \left|
    \frac{ \vrt_{  t_{ i }  + s }  }{ \tau_{  t_{ i } + s }  }
    ( b_i a_{i, s} )    -
    \frac{ \vrt_{ t_{ i } }  }{ \tau_{ t_{ i } }  }(b_i)
    \right|
    \le 2\frac{ \tau_s ( a_{i,s} ) }{ \tau_{t_i}( b_i )+
    \tau_s ( a_{i,s })}
    &\le \frac{  2N_{ i  } M_{ i  } }{ t_{ i } }
    < 2 \eps_{ i - 1 },
  \end{align*}
  So, using~\eqref{eq:good-at-somepoint}, we know that
  \[
   \frac{ \vrt_{ t_i + s } }{ \tau_{ t_i + s }} ( b_i a_{ i  , s } )
    \in B_{ 5 \eps_{i - 1} } ( \p_{i-1} )
    .
  \]

  Now let us consider the range $s = N_{ i}, \ldots, n_{i} k_{i}-1,n_{i} k_{i}$. By Lemma~\ref{lemlinalgebra} item~\eqref{eq:n+k-with-min-max} we know that 
   \[
     \min\left\{
      \frac{ \vrt_{ t_{ i } } }{ \tau_{ t_{ i } }} (b_i) ,
      \frac{ \vrt_{ s }}{ \tau_{ s }} ( a_{ i, s } ) 
    \right\}
    \le
        \frac{ \vrt_{  t_{ i } +s }}{  \tau_{  t_{ i } +s}} (b_i a_{i,s}) 
    \le
    \max\left\{
      \frac{ \vrt_{ t_{ i } } }{ \tau_{ t_{ i } } }(b_i) ,
      \frac{ \vrt_{ s } }{ \tau_{ s }} ( a_{ i, s } ) 
    \right\},
 \]
 where the inequality holds coordinate-wise. Let
 $\delta_i = \max\{ 3\eps_{i-1}, 3\eps_{i} \}$ and recall from~\eqref{eq:good-at-somepoint-again} and~\eqref{eq:large-dev-tau-2} that $(\vrt_{ t_{ i } } / \tau_{ t_{ i } } ) (b_i) \in B_{\delta_i}(\p_{i-1})$ and
 $(\vrt_{ s } / \tau_{ s } ) (a_{i,s}) \in B_{\delta_i}(\p_{i})$. Thus,
 \begin{equation}\label{eqtausss}
  \frac{ \vrt_{  t_{ i } +s }  }{ \tau_{  t_{ i } +s}  }(b_i a_{i,s}) \in R_{i}^{\delta_i}
\end{equation}
where $R_i^{\delta_i}$ is the $\delta_i$-neighbourhood of the $d$-dimensional rectangle $R_i$
that has $\p_i$ and $\p_{i+1}$ as vertices and is given by
\[
  R_i^{\delta_i} \eqdef \left\{ (x_1,\ldots,x_d) \in \R^d \colon
  \min\left\{ p_{i-1}^{(j)}, p^{(j)}_{i}\right\} - \delta_i
  < x_j < \max\left\{ p^{(j)}_{i-1}, p^{(j)}_{i}  \right\} + \delta_i
  \right\}.
 \]
  As $\lim_{i \to \infty } | \p_i - \p_{i+1} | = 0$, and as $\lim_{i \to\infty} \delta_i = 0$, we know that the rectangles $R^{\delta_i}_{i}$ accumulate precisely at the limit points $\cC$ of $(\p_i)_i$. This concludes that $\cT(x) = \cC$.
  
  To check that \eqref{eqcondtaus} holds, let $k\in\bN$ and choose $i$ such that $t_i\le k< t_{i+1}$.  Hence, it follows from \eqref{eqtausss} that
  \[
  	\left|  \frac{ \vrt_{  k }  }{ \tau_{  k }}(x) - \frac{ \vrt_{  k +1}  }{ \tau_{  k +1}}(x)\right| 
	\le2\diam(R_i^{\delta_i}).
  \]
 As $k\to\infty$, we get $i\to\infty$ and hence the latter tends to $0$. 
This implies that \eqref{eqcondtaus} holds and hence proves the lemma.
\end{proof}

\subsection{Proof of Theorem~\ref{thm:dim-H-C}}\label{sec73}

In order to conclude the proof of Theorem~\ref{thm:dim-H-C} it remains only to show that 
$\dim_{\rm H} ( \Gamma ) = 1$. This follows verbatim from the proof of Lemma~\ref{lem:dim}.
\qed

\section{Examples}
\label{sec:examples}

We consider an interval map $f \colon X \to X$ that admits a Markov partition $\cP$ which has $d\ge 2$ neutral fixed points $\{\xi_1,\ldots, \xi_d\}$. At these neutral fixed points we will always assume that we have the following expansion
\begin{equation}\label{eq:expansion}
  \text{there exist } \alpha \in (0,1), \; b_1,\ldots, b_d > 0
  \quad
  \text{such that}
  \quad
  f (x) - \xi_i
  \sim b_i ( x  - \xi_i )^{ 1 + 1 / \alpha }.
\end{equation}
We assume that there is a partition $I_1,I_2,\ldots$ of $X$ (modulo $\Leb$) into $d + d'$ open sub-intervals, where $0  \le d' \le + \infty$ such that each $I_i$ is a union of elements in $\cP$ and each restriction $f_i \eqdef f|_{I_i}$ is a $C^2$ diffeomorphism onto its image.\footnote{Notice that we do not impose that $f_i$ extends to a $C^2$ diffeomorphism on the closure of $I_i$.}

\begin{example}\label{ex:thaler}
  Suppose that $d' = 0$ and that each branch of the map is full in the sense that $f_i(I_i) = (0,1)$ and monotone increasing. Moreover, we assume that
  \begin{equation}\label{eq:uni-exp}
    | f' (x) |  > 1,\quad\text{ for all } x \not\in \{ \xi_1, \ldots, \xi_d \}
    ,
  \end{equation}
  and that Adler's distortion condition holds
  \begin{equation}\label{eq:adler}
    \sup_{1 \le i \le d} \sup_{x \in I_i} \frac{ |f'' (x)| }{ ( f'(x) )^2 }
    < \infty
    ,
  \end{equation}
  In this case we obtain a map with $d$ full orientation preserving branches, each of which contains a neutral fixed point satisfying~\eqref{eq:expansion}. In this case we define
  \begin{equation}
    Y \eqdef
    \begin{cases}
      (\gamma, f \gamma), &\text{ if } d = 2; \\
      \bigcup_{ i = 1 }^{ d } I_i \setminus f_{i}^{-1} I_i, &
      \text{ if } d \ge 3,
    \end{cases}
  \end{equation}
  where in the case $d = 2$, $\{ \gamma, f \gamma \}$ is the unique orbit of period 2. The fact that this map satisfies Assumptions~\ref{asm:m}--\ref{asm:dyn-sep} for this choice of $Y$ can be found in~\cite[Section 4.1]{CoaMelTal:2024}.
\end{example}

\begin{example}
  We can generalise the previous example by allowing an additional $1 \le d' \le +\infty$ uniformly expanding branches.   As in Example~\ref{ex:thaler}, we assume that both~\eqref{eq:uni-exp} and~\eqref{eq:adler} hold, and we set
  \[
    Y \eqdef
    \bigcup_{ i = 1 }^{ d } I_i \setminus f_{i}^{-1} I_i
    ,
  \]
  regardless of the value of $d$. 
  The fact that such maps satisfy Assumptions~\ref{asm:m}--\ref{asm:dyn-sep} can be found in~\cite[Example 5.14]{CoaMel2025-sub}.
\end{example}

\begin{example}
  \label{ex:d+D}
  Another way which can generalise Example~\ref{ex:thaler} is by allowing for critical and/or singular points at the preimages of the neutral fixed points. For simplicity, we will stick to the case that $d = 2$, which were introduced in~\cite{CoaLuzMuh2023}.
  The fact that these maps satisfy Assumptions~\ref{asm:m}--\ref{asm:dyn-sep} can be found in~\cite[Setion 4.2]{CoaMelTal:2024}.
\end{example}

\bibliographystyle{alpha}
\bibliography{bib2}

\end{document}